\documentclass{amsart}
\usepackage{amsmath}
\usepackage{amsfonts}
\usepackage{stmaryrd}
\usepackage{amssymb}
\usepackage{graphicx}
\usepackage[all]{xy}
\usepackage{pdfsync}
\usepackage{enumerate}
\usepackage{hyperref}
\usepackage{multirow} 
\usepackage{tensor} 
\usepackage{xcolor}

\newtheorem{theorem}{Theorem}[section]

\newtheorem{corol}[theorem]{Corollary}

\newtheorem{definition}[theorem]{Definition}
\newtheorem{example}[theorem]{Example}

\newtheorem{lemma}[theorem]{Lemma}

\newtheorem*{problem}{Problem}
\newtheorem{prop}[theorem]{Proposition}

\theoremstyle{definition}
\newtheorem{remark}[theorem]{Remark}

\numberwithin{equation}{section}

\newcommand{\Ss}{\mathbb S}
\newcommand{\Rr}{\mathbb R}
\newcommand{\Qq}{\mathbb Q}
\newcommand{\Zz}{\mathbb Z}

\newcommand{\Cc}{\mathbb C}

\newcommand{\Tt}{\mathbb T}

\newcommand{\eps}{\varepsilon}

\newcommand{\X}{\ensuremath{\mathfrak{X}}}

\renewcommand{\d}{\mathrm d}
\newcommand{\cG}{\ensuremath{\mathcal{G}}}
\newcommand{\cD}{\ensuremath{\mathcal{D}}}
\newcommand{\cK}{\ensuremath{\mathcal{K}}}

\newcommand{\cN}{\ensuremath{\mathcal{N}}}

\newcommand{\pr}{\ensuremath{\mathrm{pr}}}
\newcommand{\Per}{\ensuremath{\mathrm{Per}}}

\newcommand{\timesst}{\tensor[_s]{\times}{_t}}
\newcommand{\diffto}{\xrightarrow{\raisebox{-0.2 em}[0pt][0pt]{\smash{\ensuremath{\sim}}}}}

 \DeclareMathOperator{\Der}{Der}

\DeclareMathOperator{\gl}{\mathfrak{gl}}

\DeclareMathOperator{\so}{\mathfrak{so}}

\DeclareMathOperator{\uu}{\mathfrak{u}}
\renewcommand{\aa}{\mathfrak{a}}
\renewcommand{\sl}{\mathfrak{sl}}
\renewcommand{\sp}{\mathfrak{sp}}
\DeclareMathOperator{\GL}{GL}
\DeclareMathOperator{\SL}{SL}
\DeclareMathOperator{\SO}{SO}
\DeclareMathOperator{\SU}{SU}
\DeclareMathOperator{\OO}{O}
\DeclareMathOperator{\UU}{U}
\DeclareMathOperator{\Sp}{Sp}
\DeclareMathOperator{\tr}{trace}
\DeclareMathOperator{\Diff}{Diff}
\DeclareMathOperator{\Curv}{Curv}
\DeclareMathOperator{\Tors}{Tors}

\DeclareMathOperator{\rank}{rank}
\renewcommand{\hom}{\mathrm{Hom}}
\newcommand{\reg}{\mathrm{reg}}
\newcommand{\can}{\mathrm{can}}

\newcommand{\G}{\mathit{\Gamma}}            
\newcommand{\s}{\mathbf{s}}             
\renewcommand{\t}{\mathbf{t}}           
\renewcommand{\H}{\mathcal{H}}          

\newcommand{\A}{A}                      
\newcommand{\al}{\alpha}                
\newcommand{\be}{\beta}                 
\newcommand{\Lie}{\mathcal{L}}          
\renewcommand{\gg}{\mathfrak{g}}        
\newcommand{\hh}{\mathfrak{h}}          
\newcommand{\Ker}{\text{\rm Ker}\,}     
\renewcommand{\Im}{\text{\rm Im}\,}     
\newcommand{\Ad}{\text{\rm Ad}\,}       
\newcommand{\tto}{\rightrightarrows}    
\newcommand{\wmc}{\omega_{\textrm{MC}}}   
\DeclareMathOperator{\Inv}{Inv}

\newcommand{\B}{\mathrm{F}}                    
\newcommand{\BG}{\mathrm{F}_G}                 
\newcommand{\OM}{\mathrm{O}(M)}                
\newcommand{\K}{\mathcal{K}}                     
\renewcommand{\O}{\mathrm{O}}     

\begin{document}
\title[Classifying Lie Algebroid of a Geometric Structure]{The Classifying Lie Algebroid of a Geometric Structure II: $G$-structures with connection}
\author{Rui Loja Fernandes}
\address{Department of Mathematics, University of Illinois at Urbana-Champaign, 1409 W. Green Street, Urbana, IL 61801 USA}
\email{ruiloja@illinois.edu}

\author{Ivan Struchiner}
\address{Departamento de Matem\'atica, Universidade de S\~ao Paulo, Rua do Mat\~ao 1010, S\~ao Paulo, SP 05508-090, Brasil}
\email{ivanstru@ime.usp.br}

\thanks{RLF was partially supported by NSF grants DMS-1710884 and DMS-2003223.}

\begin{abstract} 
Given a $G$-structure with connection satisfying a regularity assumption we associate to it a classifying Lie algebroid. This algebroid contains all the information about the equivalence problem and is an example of a $G$-structure Lie algebroid. We discuss the properties of this algebroid, the $G$-structure groupoids integrating it and the relationship with Cartan's realization problem.
\end{abstract}

\date{\today}

\maketitle

\setcounter{tocdepth}{1}
\tableofcontents

\section{Introduction}
\label{sec:introd}

This paper is the long due second part of \cite{FernandesStruchiner1}, whose opening sentence stated ``This is the first of two papers dedicated to a systematic study of symmetries, invariants and moduli spaces of geometric structures of finite type''. The first paper was dedicated to the case of $\{e\}$-structures and a special case of Cartan's realization problem. The central object introduced in \cite{FernandesStruchiner1} was the classifying Lie algebroid of a fully regular coframe, which contains all the information relevant for the equivalence problem for coframes, i.e., $\{e\}$-structures.
Our plan for the second paper was to do a similar analysis for general $G$-structures and the general case of Cartan's realization problem. While performing this analysis we discovered a theory of $G$-structure groupoids and $G$-structure algebroids, which gives the appropriate language to deal with such problems. These objects are introduced and described in detail in our recent paper \cite{FS19}, which also gives our method to solve Cartan's realization problem. Hence, what is left from our original plan is to discuss the classifying Lie algebroid of a fully regular $G$-structure with connection and that is the main aim of this paper. 

Given a $G$-structure on a manifold $M$, denoted $\B_G(M)$, equipped with a connection $\theta\in\Omega^1(\B_G(M),\gg)$, one has the space of invariant forms $\Omega^\bullet(\B_G(M),\omega)$ consisting of all differential forms which are preserved under local equivalences of $(\B_G(M),\omega)$. The $G$-structure with connection is called \textbf{fully regular} when the space
\[ \Big\{\d_p I: I\in \Omega^0(\B_G(M),\omega)\}\subset T^*_p \B_G(M)\Big\}, \]
has constant dimension. In this case, there is naturally associated to $(\B_G(M),\omega)$ a vector bundle $A\to X$ such that
\[ \Omega^\bullet(\B_G(M),\omega)\simeq \Gamma(\wedge^\bullet A^*). \]
It follows that $A$ has a Lie algebroid structure and we call it the \textbf{classifying Lie algebroid} of the fully regular  $G$-structure with connection $(\B_G(M),\omega)$. We will see that this Lie algebroid is in fact a \emph{$G$-structure algebroid} and that it contains all the relevant information about $(\B_G(M),\omega)$, its symmetries, its invariants, etc.

The connection between Cartan's realization problems and Lie algebroids was first pointed out by Bryant in \cite{Bryant} (see also \cite{Bryant:notes}). In \cite{FS19} this is formalized precisely and it is shown that giving such a problem is the same thing as specifying a $G$-structure Lie algebroid with connection. In \cite{FS19} it is also explained how one can solve a Cartan's realization problem by integrating the $G$-structure algebroid to a $G$-structure groupoid with connection. The latter class of groupoids can be loosely described as families of $G$-structures with connection parameterized by the space of objects of the groupoid. We will see that a $G$-structure with connection in such a family need not be fully regular. When it is fully regular, we relate the $G$-structure Lie algebroid defining the family to the classifying Lie algebroid of the $G$-structure with connection.

This paper is organized as follows. In Section \ref{sec:G:structures} we review the basic facts of the theory of $G$-structures. We work with $G$-structures over effective orbifolds since these are the kind of $G$-structures that appear naturally in connection with $G$-structure groupoids and Cartan's realization problems. In Section \ref{sec:classifying:algbrd} we define the classifying Lie algebroid of a fully regular $G$-structure with connection and study its first properties. Section \ref{sec:examples} discusses several examples of $G$-structures with connection and their classifying algebroids. Section \ref{sec:G:structure:algbrd:grpd} contains a summary of the theory of $G$-structure groupoids and $G$-structure algebroids developed by us in \cite{FS19}. This is used in Section \ref{sec:G:realizations} to make precise the link between $G$-structure algebroids with connection and the classifying Lie algebroid of a single $G$-structure. Here the notion of a $G$-realization of a $G$-structure algebroid with connection plays a crucial role. We also explain in this section how $G$-realizations of a $G$-structure algebroid with connection are related to its $G$-integrations and how they can be used to codify the solutions of the associated Cartan's realization problem.

\section{$G$-structures and Connections}
\label{sec:G:structures}

In this section we give a quick review of the theory of $G$-structures. Besides fixing our notations, we also 
present a viewpoint that will be useful to us. Namely, for us $G$-structures are encoded by a free and proper $G$-action together with a ``tensorial'' 1-form (the tautological form). We then relax the freeness condition allowing for locally free $G$-actions, so we can include also $G$-structures over effective orbifolds. The main reason for allowing for locally free actions and orbifolds is not aiming at the most general approach. Rather, as we shall see, orbifolds turn out to be the natural set up for the theory.

\subsection{G-structures} We begin with the description of the classical approach to $G$-structures via reduction of the frame bundle. We then explain the transition to our approach, which does not include any reference to the base manifold or to orbifolds. 

Let $M$ be a $n$-dimensional manifold. Its \textbf{frame bundle} $\pi: \B(M) \to M$ is a $\GL_n(\Rr)$-principal bundle with fiber over $x$ the frames at $x$:
\[\pi^{-1}(x) = \{p: \Rr^n \to T_xM: p \text{ is a linear isomorphism }\}.\]
Given a closed subgroup $G \subset \GL_n(\Rr)$, a \textbf{$G$-structure on $M$} is a reduction of the frame bundle $\B(M)$ to a principal sub-bundle $\BG(M)$ whose structure group is $G$. 

A (locally defined) diffeomorphism  $\varphi: M \to N$ has a canonical lift to a (locally defined) principal bundle map between the frame bundles:
\[\widetilde{\varphi}: \B(M) \to \B(N), \quad \widetilde{\varphi}(p)(v) = \d\varphi (pv) \text{ for all } v\in\Rr^n\]
Two $G$-structures $\BG(M)$ and $\BG(N)$ are \textbf{(locally) equivalent} if there exists a (locally defined) diffeomorphism $\varphi: M \to N$ such that $\widetilde{\varphi}(\BG(M)) = \BG(N)$.


One of the main issues when dealing with $G$-structures is to characterize $G$-structures up to (local) equivalence.  For this, the \textbf{tautological form} of the $G$-structure plays a crucial role: its the $\Rr^n$-valued 1-form $\theta \in \Omega^1(\BG(M), \Rr^n)$ defined for $p\in\BG(M)$ and $\xi\in T_p \BG(M)$ by:
\[\theta_p(\xi) = p^{-1}(\d_p\pi(\xi)).\]
One then has the following classical result (see, e.g., \cite{Sternberg}):

\begin{theorem}\label{thm:tensorial-equiv}
Let $\BG(M)$ and $\BG(N)$ be two $G$-structures over $M$ and $N$ respectively. A principal bundle isomorphism
\[\xymatrix{\BG(M)\ar[d] \ar[r]^\Phi & \BG(N) \ar[d]\\
M \ar[r]_\varphi & N}\]
is an equivalence of $G$-structures if and only if $\Phi^*\theta_N = \theta_M$.
\end{theorem}

Next, it will be important for us to recognize when a principal $G$-bundle is a $G$-structure. This can be answered by identifying the main properties of the tautological 1-form. A first property is that it is a \emph{tensorial form}. To state it, we introduce the notation $\tilde{\al}\in\X(P)$ to denote the infinitesimal generator of the $G$-action:
\[\tilde{\al}_p := \frac{\d}{\d t}\Big|_{t=0}p\cdot\exp(t\al) \quad (\al\in\gg). \]
Then we have the following definition:

\begin{definition}
Let $P\to M$ be a principal $G$-bundle and $V$ a $G$-representation. A $V$-valued form $\theta\in\Omega^k(P,V)$ is \textbf{tensorial} if it is:
\begin{enumerate}[(i)]
\item \textbf{Horizontal:} $i_{\tilde{\al}}\theta=0$, for all $\al\in\gg$;
\item \textbf{$G$-equivariant:} $g^*\theta = g^{-1}\cdot \theta$, for all $g\in G$.
\end{enumerate}
\end{definition}

In the case of the tautological 1-form the relevant representation of $G \subset \GL_n(\Rr)$ is the defining representation on $\Rr^n$. Besides being a tensorial, the tautological form is also \textbf{pointwise surjective}: for each $p \in \BG(M)$, the map $\theta_p: T_p\BG(M) \to \Rr^n$ is surjective. One then has the following classical result characterizing $G$-structures among all principal $G$-bundles \cite{Bernard60}:

\begin{theorem}\label{thm:tensorial-embedding}
Let $M$ be an $n$-dimensional manifold and $G \subset \GL_n(\Rr)$ a closed Lie subgroup. Let $\pi: P \to M$ be a principal $G$-bundle and assume that $P$ comes equipped with a pointwise surjective tensorial 1-form $\theta_P \in \Omega^1(P,\Rr^n)$. Then there exists a unique embedding of principal bundles
\[\xymatrix{P \ar[rr]^\Phi \ar[dr] && \B(M) \ar[dl]\\
&M&}\]
such that $\Phi^*\theta = \theta_P$.
\end{theorem}

We conclude from the previous two results that the category of $G$-structures with equivalences can be identify with the category consisting of:
\begin{description}
\item[Objects] Pairs $(P,\theta)$ where $P$ is a manifold with a proper and free $G$-action,
\[\dim P = n + \dim G,\]
and $\theta \in \Omega^1(P, \Rr^n)$ is a tensorial, fiberwise surjective, 1-form. 
\item[Morphisms] $G$-equivariant diffeomorphisms $\Phi: P \to Q$ such that 
\[\Phi^*\theta_Q = \theta_P.\]
\end{description}
Notice that local equivalences correspond to morphisms defined on $G$-saturated open sets.

As we have already mentioned, one advantage of this approach is that one can relax the condition on the $G$-action, replacing free by \emph{locally free and effective}. The quotient $P/G$ is then an effective orbifold, and so one obtains a generalization of the theory of $G$-structures on manifolds to \textbf{$G$-structures on effective orbifolds}. In fact, the frame bundle of an effective orbifold is a manifold (see \cite{MoerdijkMrcun}) and Theorems  \ref{thm:tensorial-equiv} and \ref{thm:tensorial-embedding} remain valid in this setting. Henceforth, we consider the category above with pairs $(P,\theta)$ carrying a locally free, effective and proper $G$-action. More details on the resulting theory of $G$-structures on orbifolds will appear elsewhere \cite{DS21}.

\subsection{Connections} We will be mostly interested in $G$-structures with connections. A \textbf{connection} on a $G$-structure $(P,\theta)$ is a 1-form $\omega \in \Omega^1(P,\gg)$ which is:
\begin{enumerate}[(i)]
\item \textbf{Vertical:} $\omega(\tilde{\al}_p) = \al$ for all $\al \in\gg$;
\item \textbf{$G$-equivariant:} $g^*\omega = \mathrm{Ad}_{g^{-1}}\circ \omega$ for all $g \in G$.
\end{enumerate}

When working in the category of $G$-structures with connections it is natural to consider \textbf{connection preserving equivalences}. Therefore, the morphisms between two $G$-structures with connection $(P, \theta_P, \omega_P)$ and $(Q, \theta_Q, \omega_Q)$ are diffeomorphisms $\Phi: P \to Q$ which are $G$-equivariant and satisfy: 
\[ \Phi^*\theta_Q = \theta_P,\quad \Phi^*\omega_Q = \omega_P. \]

If $(P, \theta, \omega)$ is a $G$-structure with connection, for each $p \in P$ we have a linear isomorphism: 
\[(\theta,\omega)_p : T_pP \to \Rr^n\oplus \gg.\]
In other words, the pair $(\theta,\omega)$ may be interpreted as a coframe on $P$. The exterior derivatives of $\theta$ and $\omega$ can then be re-expressed in terms of the coframe $(\theta,\omega)$. This yields the well-known {\bf structure equations}:
\begin{equation}
\label{eq:structure} \left\{
\begin{array}{l}
\d \theta=T ( \theta \wedge \theta) - \omega \wedge \theta\\
\\
\d \omega= R ( \theta \wedge \theta) - \omega\wedge\omega
\end{array}
\right.
\end{equation}
where $T: P \to \hom(\wedge^2\Rr^n, \Rr^n)$ is the torsion and $R : P \to \hom(\wedge^2\Rr^n,\gg)$ is the curvature of the connection $\omega$. These functions are both $G$-equivariant for the induced actions on the vector spaces $\hom(\wedge^2\Rr^n, \Rr^n)$ and  $\hom(\wedge^2\Rr^n,\gg)$. Note that in the expressions above the wedge products make use of the of the action of $\gg\subset\gl_n(\Rr)$ on $\Rr^n$ and of the Lie bracket of $\gg$:
\begin{align*}
\omega \wedge \theta\, (\xi_1,\xi_2)&=\omega(\xi_1)\cdot\theta(\xi_2)-\omega(\xi_2)\cdot\theta(\xi_1)\\
\omega \wedge \omega\, (\xi_1,\xi_2)&=[\omega(\xi_1),\omega(\xi_2)].
\end{align*}

There may or may not exist a torsion-free connection ($T=0$) on a $G$-structure, and if it exists it might not be unique. The uniqueness, or lack thereof, is controlled by $G$ and its defining action on $\Rr^n$, and does not depend on the $G$-structure itself. On the other hand, the existence of torsion-free connection, in general, depends on the $G$-structure and is a first obstruction to integrability, i.e., for the $G$-structure to be locally equivalent to the trivial $G$-structure. For the proofs of these facts and more details we refer to \cite{Kobayashi,MolinoGstructures,Sternberg}. The table below gives some classes of $G$-structures and the geometric significance of the existence of a torsion-free connection.

\begin{figure}[!h]
\makebox[1 \textwidth][c]{       
\resizebox{1.1 \textwidth}{!}{   
\begin{tabular}{||r|c|c|c||}\hline
$G$            & \textbf{Geometric Structure} &\textbf{Isomorphisms} & \textbf{Torsion-Free Connection} \\ \hline \hline
$\GL_n^+(\Rr) \subset \GL_n(\Rr)\ $ & \footnotesize{Orientation} & \footnotesize{Orientation Preserving Diffeomorphisms} & \footnotesize{All}\\ 
\hline $\mathrm{SL}_{n}(\Rr) \subset \GL_n(\Rr)\ $ & \footnotesize{Volume Form} & \footnotesize{Volume Preserving Diffeomorphisms}    & \footnotesize{All}       \\
\hline $\mathrm{Sp}_{n}(\Rr) \subset \GL_{2n}(\Rr)$  & \footnotesize{Almost Symplectic Structure} & \footnotesize{Symplectomorphisms}   & \footnotesize{Symplectic Structures}   \\ 
\hline $\mathrm{O}_{n}(\Rr) \subset \GL_{n}(\Rr)\ $  & \footnotesize{Riemannian Metrics} & \footnotesize{Isometries}   & \footnotesize{All}   \\ 
\hline $\mathrm{GL}_{n}(\Cc) \subset \GL_{2n}(\Rr)$  & \footnotesize{Almost Complex Structures} & \footnotesize{Holomorphic Diffeomorphisms}   & \footnotesize{Complex Structures}   \\ 
\hline $\mathrm{U}_n\subset \GL_{2n}(\Rr)$  & \footnotesize{Hermitian Metrics} & \footnotesize{Hermitian Isometries}  & \footnotesize{K\"ahler Structures}\\
\hline $\{1\}\subset \GL_{n}(\Rr)\ $  & \footnotesize{Coframes} & \footnotesize{Equivalence of Coframes}  & \footnotesize{Coordinate System Coframes}    \\  \hline
\end{tabular}}}
\end{figure}

\section{The classifying algebroid}
\label{sec:classifying:algbrd}

In this section we associate to a \emph{fully regular} $G$-structure with connection a Lie algebroid. This algebroid contains all the relevant information to decide if two $G$-structures with connection are equivalent.

\subsection{Invariants}

The solution to the equivalence problem relies on understanding the invariants of the $G$-structure with connection.

\begin{definition}
Let $(P,\theta,\omega)$ be a $G$-structure with connection. A differential form $\omega\in\Omega^k(P)$ is called an 
\textbf{invariant $k$-form} of $(P,\theta,\omega)$ if for any locally defined self equivalence $\Phi:\pi^{-1}(U)\to \pi^{-1}(U')$ one has
\[ \Phi^*\Omega=\Omega. \]
The space of invariant $k$-forms is denoted $\Omega^k(P,\theta,\omega)$.
\end{definition}

When $k=0$, the elements of $\Inv(P,\theta,\omega):=\Omega^0(P,\theta,\omega)$ are the \textbf{invariant functions} of the $G$-structure with connection. 

More generally, it is useful to consider invariant functions and forms with values in a vector space $V$. Their components relative to any basis for $V$ will still be ordinary invariant functions and forms. 

\begin{example}
For an arbitrary $G$-structure with connection $(P,\theta,\omega)$ the torsion $T$ is an invariant function with values in the vector space $\hom(\wedge^2\Rr^n, \Rr^n)$ and the curvature $R$ is an invariant function with values in the vector space $\hom(\wedge^2\Rr^n,\gg)$.
\end{example} 

Actually, the torsion and the curvature are examples of \emph{$G$-equivariant} invariant functions. In general, when $V$ is a representation of $G$ one can talk about a $V$-valued \textbf{$G$-equivariant invariant $k$-form} $\Omega\in\Omega^k(P,V)$, i.e., invariant $k$-forms with values in $V$ which are 
$G$-equivariant:
\[ g^*\Omega=g^{-1}\cdot \Omega \quad (g\in G). \]

\begin{example}
For an arbitrary $G$-structure with connection $(P,\theta,\omega)$ the tautological 1-form $\theta$ is a $G$-equivariant invariant 1-form with values in $\Rr^n$ and the connection 1-form $\omega$ is a $G$-equivariant invariant 1-form with values in $\gg$.
\end{example} 

\begin{remark}
\label{rem:equivariant:invariants}
One can also express $V$-valued $G$-equivariant invariant $k$-forms as sections of vector bundles. Given a $G$-representation $V$, one can form the associated vector bundle $E^k:=(\wedge^kTP\times_G V)\to M$, and one obtains a 1:1 correspondence:
\[ 
\left\{\txt{$G$-equivariant forms\\ 
$\Omega\in\Omega^k(P,V)$\\ \,}\right\}
\tilde{\longleftrightarrow}
\left\{\txt{sections \\ 
$s\in\Gamma(E^k)$ \\ \,} \right\}.
\]
To express the invariance condition one observes that for any local equivalence $\Phi:\pi^{-1}(U)\to \pi^{-1}(U')$ one obtains a vector bundle map:
\[
\vcenter{\xymatrix{E^k|_U\ar[d] \ar[r]^{\Phi_{E^k}} & E^k|_{U'} \ar[d]\\
U \ar[r]_\varphi & U'}}\qquad 
[(\xi,v)]\mapsto [(\d\Phi(\xi),v)]. 
\]
Hence, one may call $s\in \Gamma(E^k)$ an \textbf{invariant section} if for any local equivalence $\Phi:\pi^{-1}(U)\to \pi^{-1}(U')$ the section satisfies:
\[ s\circ \phi=\Phi_{E^k}\circ s. \]
Then one finds that there is a 1:1 correspondence:
\[ 
\left\{\txt{$G$-equivariant invariant\\ 
forms $\Omega\in\Omega^k(P,V)$\\ \,}\right\}
\tilde{\longleftrightarrow}
\left\{\txt{invariant sections \\ 
$s\in\Gamma(E^k)$ \\ \,} \right\}.
\]
\end{remark}

A very elementary but important fact is that the space of invariant forms is a subcomplex of the de Rham complex:

\begin{prop}
For any invariant form $\Omega\in\Omega^k(P,\theta,\omega)$ its differential is also an invariant form:  $\d\Omega\in\Omega^{k+1}(P,\theta,\omega)$. If $V$ is a $G$-representation and $\Omega$ is $G$-equivariant $V$-valued invariant form, so is $\d\Omega$.
\end{prop}

\begin{proof}
The differential commutes with pullbacks.
\end{proof}

In the case of an invariant function $I\in\Inv(P,\theta,\omega)$ one can express its differential in terms of the coframe $(\theta, \omega)$ on $P$:
\[\d I = \frac{\partial I}{\partial \theta} \theta + \frac{\partial I}{\partial \omega} \omega,\]
where the coefficients are vector valued maps:
\[\frac{\partial I}{\partial \theta}: P \to \hom(\Rr^n, \Rr),\quad \frac{\partial I}{\partial \omega}: P \to \hom(\gg,\Rr).\]
By the previous proposition, $\d I$ is an invariant form, and since $\theta$ and $\omega$ are also invariant forms, one deduces that the coefficients are invariant functions.

\begin{definition}
The map $(\frac{\partial I}{\partial \theta},\frac{\partial I}{\partial \omega}): P \to \hom(\Rr^n\oplus\gg,\Rr)$ is called the \textbf{coframe derivative of I with respect to $(\theta, \omega)$}. 
\end{definition}

Notice that if we had started with a $G$-equivariant invariant function $I:P\to V$ then one would obtain a coframe derivative
\[ \Big(\frac{\partial I}{\partial \theta},\frac{\partial I}{\partial \omega}\Big): P \to \hom(\Rr^n\oplus\gg,V) \]
which is also $G$-equivariant with respect to the induced $G$-action on $\hom(\Rr^n\oplus\gg,V)$.

\begin{remark}\label{rmk:covariant:derivative}
As mentioned in Remark \ref{rem:equivariant:invariants}, to a $G$-equivariant invariant corresponds a section $s_I$ of the associated vector bundle $E=(P\times V)/G$. This bundle inherits also a linear connection $\nabla$ from the connection $\omega$ and the coframe derivative $\frac{\partial I}{\partial \theta}$ corresponds to the covariant derivative $\nabla s_I$. Similarly, one finds that the coframe derivative $\frac{\partial I}{\partial \omega}$ corresponds to an algebroid derivative  $\widetilde{\nabla} s_I$ along sections of the adjoint bundle $\Ad(P):=(P\times \gg)/G$, a bundle of Lie algebras. Here, $\widetilde{\nabla}: \Gamma(\Ad(P)) \times \Gamma(E) \to \Gamma(E)$ is the representation
\[\widetilde{\nabla}_{\sigma} s_I = s_{\Lie_{\widetilde{\sigma}}I},\]
where $\widetilde{\sigma} \in \X(P)$ is the $G$-invariant, vertical vector field induced by $\sigma \in \Gamma(\Ad(P))$.
\end{remark}

More generally, one can also express any invariant form $\Omega\in\Omega^k(P,\theta,\omega)$ in terms of the basis of $k$-forms provided by the exterior powers of the coframe $(\theta,\omega)$, and then express its differential $\d\Omega$ similarly in terms of $(\theta,\omega)$. We will not need this more general coframe differentiation procedure since we will soon see a more natural approach using Lie algebroid theory.

\subsection{Fully regular $G$-structures with connection}

Iterating the process of differentiating known invariants one can obtain an infinite list of invariants of the $G$-structure with connection. These invariants may not be independent and, in general, there will exist relations among them. One finds two types of relations:
\begin{itemize}
\item ``universal relations" which, for a fixed $G$, are the same for all $G$-structures with connections. For
example, consequences of the fact that $\d^2 = 0$.
\item ``special relations" arising from the geometry of the particular $G$-structure with connection.
\end{itemize} 

For an example of a ``universal relation'', consider an $\mathrm{O}_n(\Rr)$-structure $(P, \theta, \omega)$, where $\omega$ is the unique torsion free connection on $P$, i.e., the Levi-Civita connection. Then the curvature $R$ satisfies the Bianchi identity, which is already a universal relation. Moreover, this identity implies that the covariant derivative of the curvature: 
\[\frac{\partial R}{\partial \theta}: P \to \hom(\Rr^n, \hom(\wedge^2\Rr^n, \mathfrak{o}_n))\] 
actually takes values in the subspace $\mathcal{K} \subset \hom(\Rr^n, \hom(\wedge^2\Rr^n, \mathfrak{o}_n))$ given by:
\[\mathcal{K} = \Big\{K : K(u)(v,w) + K(v)(w,u)+K(w)(u,v) = 0, \forall u,v,w \in \Rr^n\Big\}. \]
We will see examples of ``special relations'' later.

Obviously, if $\Phi:(P,\theta_P, \omega_P) \to (Q, \theta_Q, \omega_Q)$ is an equivalence between $G$-structures with connections we obtain an isomorphism of complexes: 
\[ \Phi^*:(\Omega^k(Q, \theta_Q,\omega_Q),\d)\diffto(\Omega^k(P,\theta_P,\omega_P),\d). \]
For example, a necessary condition for equivalence is that the structure functions and its coframe derivatives of any order must correspond under the equivalence. This gives an infinite set of necessary conditions for equivalence, that would be hard to work with. However, observe that if invariants $I_{1},...,I_{l}\in\Inv(Q,\theta_Q,\omega_Q)$ satisfy a functional relationship 
\[ F(I_{1},...,I_l)=0, \]
then the corresponding elements $\Phi^*I_1, \ldots, \Phi^*I_l$ in $\Inv(P, \theta_P,\omega_P)$ must satisfy the same functional relation
\begin{equation*}
F\left(\Phi^*I_1, \ldots, \Phi^*I_l \right)=0.
\end{equation*}
This shows that we do not need to deal with all the invariant functions in $\Inv(P,\theta,\omega)$, but only with those that are functionally independent.

\begin{definition}
A $G$-structure with connection $(P,\theta,\omega)$ is called {\bf fully regular} if the dimension of the space:
\[ \mathcal{I}_p = \{ \d_p I: I\in \Inv(P,\theta,\omega)\}\subset T_p^*P, \]
does not vary with $p$. This dimension is called the {\bf rank} of the $G$-structure. 
\end{definition}

Given a $G$-structure with connection $(P,\theta,\omega)$, we will say that two points $p,q\in P$ are {\bf locally equivalent} if there exist a local equivalence $\Phi: U_p \to U_q$ such that $\Phi(p) = q$. This defines 
an equivalence relation $\sim$ on $P$. Obviously, any invariant $I:P\to \Rr$ descends to a map on the quotient space $I:P/\sim\to \Rr$. In general, the quotient $P/\sim$ is not a nice space. However, in the fully regular case we find:

\begin{prop}
Let $(P,\theta,\omega)$ be a fully regular $G$-structure with connection $(P,\theta,\omega)$ of rank $d$. The quotient
\[X_{(\theta,\omega)} := P{/}\sim\]
has a smooth structure of dimension $d$ such that the quotient map $h: P \to X_{(\theta,\omega)}$ is a submersion. Moreover, there is a smooth proper $G$-action on $X_{(\theta,\omega)}$ for which $h: P \to X_{(\theta,\omega)}$ is $G$-equivariant.
\end{prop}

\begin{proof}
The fact that $X_{(\theta,\omega)}$ is a manifold follows from \cite[Prop 2.7]{FernandesStruchiner1} applied to the coframe $(\theta,\omega)$.  Observe that if $p\sim q$ and $g\in G$ then $pg\sim qg$, so the $G$-action descends to $X_{(\theta,\omega)}$ making the quotient map $h: P \to X_{(\theta,\omega)}$ a $G$-equivariant map. Since this map is a submersion, the action of $G$ on $X_{(\theta,\omega)}$ is smooth. Since the $G$-action on $P$ is proper, so is the $G$-action on $X_{(\theta,\omega)}$.
\end{proof}

\begin{definition}
The manifold $X_{(\theta,\omega)}$ is called the \textbf{classifying manifold} and the map $h: P \to X_{(\theta,\omega)}$ is called the \textbf{classifying map} of the fully regular $G$-structure with connection $(P,\theta,\omega)$.
\end{definition}

Obviously, one has an identification between invariant functions and functions on the classifying manifold:
\[ h^*:C^\infty(X_{(\theta,\omega)}) \diffto \Inv(\theta,\omega). \]
On the other hand, for any $G$-representation $V$, one can identify the $G$-equivariant invariant maps $I:P\to V$ with the $G$-equivariant smooth maps $X_{(\theta,\omega)}\to V$. 

We will see later other reasons for referring to $X_{(\theta,\omega)}$ as the \emph{classifying} manifold.

\subsection{Classifying algebroid}

We saw in the previous paragraph that for a fully regular $G$-structure with connection $(P,\theta,\omega)$ the invariant functions are in 1:1 correspondence with the functions on the classifying manifold $X_{(\theta,\omega)}$. It is natural to wonder what objects corresponds to the invariant differential forms $\Omega^k(P,\theta,\omega)$.

For this we observe that the coframe $(\theta,\omega)$ yields a vector bundle map
\[
\xymatrix@C=15pt{TP\ar[d] \ar[rr]^---{H=(\theta,\omega)} & & A_{(\theta,\omega)} \ar[d]\\
P \ar[rr]_h && X_{(\theta,\omega)}}
\]
where $A_{(\theta,\omega)}:=X_{(\theta,\omega)}\times(\Rr^n\oplus\gg)\to X_{(\theta,\omega)}$ is the trivial vector bundle. Since $H$ is a fiberwise isomorphism we obtain an injective pullback map
\[ H^*:\Omega^k(A_{(\theta,\omega)})\to \Omega^k(P), \]
where $\Omega^k(A_{(\theta,\omega)})$ denotes the sections of the bundle $\wedge^k A_{(\theta,\omega)}^*\to X_{(\theta,\omega)}$.

\begin{prop}
The image of the pullback map $H^*$ are the invariant forms, so one has an isomorphism:
\[ H^*:\Omega^k(A_{(\theta,\omega)}) \diffto \Omega^k(P,\theta,\omega). \]
\end{prop}

\begin{proof}
Let $\Psi:P\to P$ be a local equivalence. 
Since
\[ h\circ\Psi=h,\quad \Psi^*\theta=\theta,\quad \Psi^*\omega=\omega, \]
it follows that
\[ H\circ \d\Psi=H. \]
Hence, if $\Omega\in\Omega^k(A_{(\theta,\omega)})$ we have:
\[ \Psi^*H^*\Omega=H^*\Omega. \]
Since $\Psi$ was an arbitrary local equivalence, we conclude that $H^*\Omega\in \Omega^k(P,\theta,\omega)$.

Conversely, any invariant form $\widetilde{\Omega}\in\Omega^k(P,\theta,\omega)$ can be expressed in terms of the base of $k$-forms provided by the coframe $(\theta,\omega)$ as
\[ \widetilde{\Omega}=\sum_{\substack{i_1,\dots,i_n\\j_1\dots,j_r}} \widetilde{\Omega}_{i_1,\dots,i_n,j_1\dots,j_r}
\theta^{i_1}\wedge\cdots\wedge\theta^{i_n}\wedge\omega^{j_1}\wedge\cdots\wedge\omega^{j_r}, \]
where the coefficients are invariant functions. Then we must have:
\[ \widetilde{\Omega}_{i_1,\dots,i_n,j_1\dots,j_r}=\Omega_{i_1,\dots,i_n,j_1\dots,j_r}\circ h, \]
for unique smooth functions $\Omega_{i_1,\dots,i_n,j_1\dots,j_r}$ on $X_{(\theta,\omega)}$. Hence, we conclude that $\widetilde{\Omega}=H^*\Omega$, with
\[ \Omega:=\sum_{\substack{i_1,\dots,i_n\\j_1\dots,j_r}} \widetilde{\Omega}_{i_1,\dots,i_n,j_1\dots,j_r}
u^{i_1}\wedge\cdots\wedge u^{i_n}\wedge \al^{j_1}\wedge\cdots\wedge \al^{j_r}, \]
where $\{u^i,\al^j\}\subset\Omega^1(A_{(\theta,\omega)})$ is the unique base of trivializing sections of $A_{(\theta,\omega)}^*$ such that $H^*u^i=\theta^i$ and $H^*\al^j=\omega^i$.
\end{proof}

It follows from the previous proposition that there is a linear differential operator
\[ \d_A:\Omega^\bullet(A_{(\theta,\omega)})\to \Omega^{\bullet+1}(A_{(\theta,\omega)}), \]
characterized by:
\[ H^*\d_A=\d H^*. \]
Notice that $\d_A$ is a differential on $\Omega^\bullet(A_{(\theta,\omega)})$:
\[
\d_A^2=0, \qquad \d_A (\eta\wedge\mu)=\d_A\eta\wedge\mu+(-1)^{\deg\eta}\eta\wedge\d_A\mu.
\]
This means that the vector bundle $A_{(\theta,\omega)}\to X_{(\theta,\omega)}$ carries a Lie algebroid structure. The anchor $\rho_A$ is defined on a section $s\in\Gamma(A_{(\theta,\omega)})$ by:
\begin{equation}
\label{eq:anchor:differential} 
\Lie_{\rho_A(s)} f=\langle s,\d_A f\rangle\qquad (f\in C^\infty(X_{(\theta,\omega)})),
\end{equation}
and the Lie bracket $[\cdot,\cdot]_A$ of sections $s_1,s_2\in\Gamma(A_{(\theta,\omega)})$ is determined by requiring that for any $\eta\in \Omega^1(A_{(\theta,\omega)})$:
\begin{equation}
\label{eq:bracket:differential} 
\langle [s_1,s_2]_A,\eta\rangle=\Lie_{\rho(s_1)} \langle s_2,\eta\rangle-\Lie_{\rho(s_2)} \langle s_1,\eta\rangle
-\d_A\eta(s_1,s_2),
\end{equation}
For background on Lie algebroids and Lie groupoids see \cite{CrainicFernandes:lectures}.

\begin{definition}
Let $(P,\theta,\omega)$ be a fully regular $G$-structure with connection. The Lie algebroid $(A_{(\theta,\omega)},\rho_A,[\cdot,\cdot]_A)$ is called the \textbf{classifying Lie algebroid} of $(P,\theta,\omega)$.
\end{definition}

In the next paragraph we will give a more explicit form of the classifying Lie algebroid. For now we observe that, by definition, $H=(\theta,\omega)$ intertwines the de Rham differential and the Lie algebroid differential. Hence, we have:

\begin{corol}
\label{cor:classify:algbrd}
If $(P,\theta,\omega)$ is a fully regular $G$-stucture, then the coframe $(\theta,\omega)$ together with $h:P\to X_{(\theta,\omega)}$ form a Lie algebroid map:
\[
\xymatrix@C=15pt{TP\ar[d] \ar[rr]^---{H=(\theta,\omega)} & & A_{(\theta,\omega)} \ar[d]\\
P \ar[rr]_h && X_{(\theta,\omega)}}
\]
\end{corol}

The previous corollary shows, in particular, that $H$ preserves the anchors:
\[ \rho_A\circ H=\d h. \]
Since $h:P\to X_{(\theta,\omega)}$ is a submersion, this implies that $\rho_A$ is surjective. In other words, the classifying Lie algebroid of a fully regular $G$-structure with connection is a \emph{transitive Lie algebroid}.

\subsection{Canonical form of the classifying algebroid}
\label{sec:canonical:form}

Let $(P,\theta,\omega)$ be a fully regular $G$-structure with connection. The torsion and the curvature of the connection, being $G$-equivariant invariant maps, determine $G$-equivariant maps on the classifying manifold:
\[ T: X_{(\theta,\omega)} \to \hom(\wedge^2\Rr^n, \Rr^n), \qquad R : X_{(\theta,\omega)} \to \hom(\wedge^2\Rr^n,\gg).\]
Denote the infinitesimal $G$-action on the classifying manifold by $\psi:\gg\to \X(X_{(\theta,\omega)})$. If we identify functions on $X_{(\theta,\omega)}$ with invariant functions on $P$, then this infinitesimal action amounts to the coframe derivative $\frac{\partial }{\partial \omega}$:
\begin{equation}
\label{eq:g:coframe:derivative} 
h^*(\Lie_{\psi(\al)} f)=\frac{\partial(f\circ h)}{\partial \omega}(\al)\qquad (\al\in\gg).
\end{equation}
We can also view the infinitesimal action as a bundle map:
\[ \psi:X_{(\theta,\omega)}\times\gg\to TX_{(\theta,\omega)}, \]
which is $G$-equivariant:
\[ \psi(x,\alpha)\cdot g=\psi(xg,\Ad_{g^{-1}}\alpha), \]
for any $x\in X_{(\theta,\omega)}$, $\al\in\gg$ and $g\in G$.

Similarly, there is a linear map $F:\Rr^n\to\X(X_{(\theta,\omega)})$ associated with the coframe derivative $\frac{\partial }{\partial \theta}$:
\begin{equation}
\label{eq:Rn:coframe:derivative} 
h^*(\Lie_{F(u)} g) =\frac{\partial (f\circ h)}{\partial \theta}(u)\qquad (u\in\Rr^n).
\end{equation}
Again, we can view this map as a bundle map:
\[ F:X_{(\theta,\omega)}\times\Rr^n\to TX_{(\theta,\omega)}, \]
which is $G$-equivariant:
\[ F(x,v)\cdot g=F(xg,g^{-1}\cdot v), \]
for any $x\in X_{(\theta,\omega)}$, $v\in\Rr^n$ and $g\in G$. Note, however, that $F$ is not associated with any infinitesimal action of $\Rr^n$. Instead, we have:

\begin{prop}
\label{prop:classify:algbrd}
If $(P,\theta,\omega)$ is a fully regular $G$-structure with connection, then the structure maps of its classifying Lie algebroid $A_{(\theta,\omega)}$ take the following form:
\begin{enumerate}[(i)]
\item The anchor $\rho:A_{(\theta,\omega)}\to TX_{(\theta,\omega)}$ is the bundle map $\rho(u,\al)  = F(u) + \psi(\al)$;
\item The bracket is defined on constant sections $(u,\al), (v, \be) \in \Gamma(A_{(\theta,\omega)})$ by
\[[(u,\al), (v, \be)] = (\al\cdot v - \be\cdot u - T(u,v), [\al,\be]_\gg - R(u,v)).\]
\end{enumerate}
\end{prop}

\begin{proof}
First, from expression \eqref{eq:anchor:differential} for the anchor, we find that:
\begin{align*}
h^*(\Lie_{\rho_A(s)} f) &=h^*\langle s,\d_A f\rangle= \langle H^*s,H^*\d_A f\rangle
=\langle H^*s,\d H^* f\rangle\\
&=\Big\langle H^*s,  \frac{\partial (f\circ H)}{\partial \theta} \theta + \frac{\partial (f\circ H)}{\partial \omega} \omega\Big\rangle
\end{align*}
Hence, if $s=(u,\al)$ is a constant section with $u\in\Rr^n$ and $\al\in\gg$, using \eqref{eq:g:coframe:derivative}  and \eqref{eq:Rn:coframe:derivative} we find:
\begin{align*} 
h^*(\Lie_{\rho_A(u+\al)} f)&=\Big\langle H^*(u,\al), \frac{\partial (f\circ H)}{\partial \theta} \theta + \frac{\partial (f\circ H)}{\partial \omega} \omega\Big\rangle\\
&=\frac{\partial(f\circ h,)}{\partial \theta}(u)+\frac{\partial(f\circ h)}{\partial \omega}(\al)\\
&=h^*(\Lie_{F(u)} f)+h^*(\Lie_{\psi(\al)} f)=h^*(\Lie_{F(u)+\psi(\al)} f).
\end{align*}
This proves the expression for the anchor.

To find the expression for the Lie bracket, fix a basis $\{u_i\}$ for $\Rr^n$ and a basis $\{\al_j\}$ for $\gg$. We denote by the same letters the corresponding constant sections of $A_{(\theta,\omega)}$ and with a lower index the corresponding dual sections $u^i,\al^j\in\Omega^1(A_{(\theta,\omega)})$. Notice that $H^*u^i=\theta^i$ and $H^*\al^j=\omega^j$, the components of the tautological and connections form relative to the fixed basis. Then, from expression \eqref{eq:bracket:differential} for the bracket we find:
\begin{align*} 
h^*\langle [(u,\al),(v, \be)]_A,u^i\rangle&=-h^*(\d_Au^i((u,\al),(v, \be)))\\
&=-(H^*\d_Au^i)(H^*(u,\al),H^*(v,\be))\\
&=-(\d H^*u^i)(H^*(u,\al),H^*(v,\be))\\
&=-(\d\theta^i)(H^*(u,\al),H^*(v,\be))\\
&=-(T ( \theta \wedge \theta) - \omega \wedge \theta)^i(H^*(u,\al),H^*(v,\be))\\
&=h^*(\al\cdot v - \be\cdot u - T(u,v))^i=h^*\langle \al\cdot v - \be\cdot u - T(u,v),u^i\rangle,
\end{align*}
where we used the 1st structure equation \eqref{eq:structure}. Similarly, we find:
\begin{align*} 
h^*\langle [(u,\al),(v, \be)]_A,\al^j\rangle&=-h^*(\d_A\al^j((u,\al),(v, \be)))\\
&=-(H^*\d_A\al^j)(H^*(u,\al),H^*(v,\be))\\
&=-(\d H^*\al^j)(H^*(u,\al),H^*(v,\be))\\
&=-(\d\omega^j)(H^*(u,\al),H^*(v,\be))\\
&=-(R ( \theta \wedge \theta) - \omega\wedge\omega)^j(H^*(u,\al),H^*(v,\be))\\
&=h^*([\al,\be]_\gg - R(u,v))^j=h^*\langle [\al,\be]_\gg - R(u,v),\al^j\rangle,
\end{align*}
where we used the 2nd structure equation \eqref{eq:structure}.
\end{proof}

We will refer to the expressions for the bracket and anchor given in the previous proposition as the \textbf{canonical form} of the classifying algebroid. This explicit form shows that the action algebroid associated with the infinitesimal $\gg$-action on $X_{(\theta,\omega)}$ is a subalgebroid of the classifying algebroid:
\[ X_{(\theta,\omega)}\times\gg\to A_{(\theta,\omega)}, (x,\al)\mapsto (x,(0,\al)). \]

The canonical form also shows that we have a $G$-action on the classifying algebroid by algebroid automorphisms, covering the $G$-action on $X_{(\theta,\omega)}$:
\[
\vcenter{
\xymatrix@C=15pt{G\times A_{(\theta,\omega)}\ar[d] \ar[r]  & A_{(\theta,\omega)} \ar[d]\\
G\times X_{(\theta,\omega)}  \ar[r] & X_{(\theta,\omega)}}}
\qquad (x,u,\al)\cdot g:=(xg,g^{-1}\cdot u,\Ad_{g^{-1}}\al).
\]
In particular, the coframe together with the classifying map give a \emph{$G$-equivariant} Lie algebroid map:
\[
\xymatrix@C=15pt{TP\ar[d] \ar[rr]^---{H=(\theta,\omega)} & & A_{(\theta,\omega)} \ar[d]\\
P \ar[rr]_h && X_{(\theta,\omega)}}
\]

\subsection{Symmetries}
We have the following natural notions of symmetries:

\begin{definition}
Let $(P,\theta,\omega)$ be a $G$-structure with connection. A \textbf{symmetry} of $(P,\theta,\omega)$ is a self-equivalence $\Phi:P\to P$. An \textbf{infinitesimal symmetry} of $(P,\theta,\omega)$ is a vector field $\xi \in \X(P)$ such that $\Lie_{\xi}\theta=\Lie_{\xi}\omega=0$. A \textbf{germ of infinitesimal symmetry} at $p$ is a germ of a vector field $\xi \in \X(U)$ defined in a neighborhood $U$ of $p$ that satisfies $\Lie_{\xi}\theta|_U=\Lie_{\xi}\omega|_U=0$
\end{definition}

The infinitesimal symmetries of a $G$-structure with connection $(P,\theta,\omega)$ form a Lie subalgebra $\X(P,\theta,\omega)\subset\X(P)$. The germs of infinitesimal symmetry at $p$ also form a Lie algebra, denoted $\X(P,\theta,\omega)_p$. There is an injective restriction homomorphism -- see \cite[Lemma 5.9]{FernandesStruchiner1}:
\[ \X(P,\theta,\omega)\to \X(P,\theta,\omega)_p. \]
On the global side, the symmetries of $(P,\theta,\omega)$ form a subgroup $\Diff(P,\theta,\omega)\subset \Diff(P)$ and we have the following classical result -- see, e.g., \cite[Thm I.5.1]{Kobayashi}:

\begin{theorem}
The group $\Diff(P,\theta,\omega)$ of symmetries of a $G$-structure with connection is a Lie group with Lie algebra the subspace of  $\X(P,\theta,\omega)$ generated by the complete vector fields.
\end{theorem}

For a general $G$-structure with connection it may be hard to relate the Lie algebra of infinitesimal symmetries $\X(P,\theta,\omega)$ and the Lie algebras of germs of infinitesimal symmetries $\X(P,\theta,\omega)_p$. Moreover, the later can depend on the point $p\in P$. 

However, for a fully regular $G$-structure with connection one can relate their infinitesimal symmetries to the isotropy Lie algebra of its classifying algebroid $A_{(\theta,\omega)}$. Since this is a transitive algebroid, its isotropy Lie algebras are all isomorphic. 

\begin{prop} 
\label{prop:inf:sym:coframe}
Let $(P,\theta,\omega)$ be a fully regular $G$-structure with connection and let $A_{(\theta,\omega)} \to X_{(\theta,\omega)}$ be its classifying Lie algebroid. Then the Lie algebra of germs of infinitesimal symmetries $\X(P,\theta,\omega)_p$ is isomorphic to the isotropy Lie algebra $\ker\rho_{h(p)}\subset A_{(\theta,\omega)}$. In particular,
\[  \dim \X(P,\theta,\omega)_p = \dim P - \dim X_{(\theta,\omega)}. \]
\end{prop}

\begin{proof}
The proof follows from \cite[Prop 5.7]{FernandesStruchiner1} applied to the coframe $(\theta,\omega)$.
\end{proof}

Although for a fully regular $G$-structure with connection the Lie algebras $\X(P,\theta,\omega)_p$ are all isomorphic, we will see in the examples in the next section that, in general, one still has strict inclusions:
\[ \textrm{Lie}(\Diff(P,\theta,\omega)) \subsetneq \X(P,\theta,\omega)\subsetneq \X(P,\theta,\omega)_p. \]
We will give later conditions under which all these Lie algebras coincide and, moreover, the natural action of $\Diff(P,\theta,\omega)$ on $P$ is a proper and free action, with orbits the fibers of the classifying map $h$. When this is the case, the classifying map
\[ h:P\to X_{(\theta,\omega)} \] 
is a principal $\Diff(P,\theta,\omega)$-bundle and it follows that the classifying algebroid is isomorphic to the Atiyah algebroid of this principal bundle:
\[ A_{(\theta,\omega)}\simeq TP/H, \quad H:=\Diff(P,\theta,\omega). \]
In particular, we obtain a Lie groupoid integrating $A_{(\theta,\omega)}$, namely the gauge groupoid of the principal bundle $h:P\to X_{(\theta,\omega)}$:
\[ \G_{(\theta,\omega)}:=(P\times_HP)\tto X_{(\theta,\omega)}. \]
The $\s$-fibers of this groupoid are copies of $P$ and hence are themselves $G$-structures (with connections). This is an example of a \emph{$G$-structure groupoid} (with connection), which will be study in the next sections.

\section{Examples}
\label{sec:examples}

\subsection{Non-fully regular $G$-structures} We start by giving a class of $G$-structures which are not fully regular.

Consider a Riemannian manifold $(M,\eta)$. Its orthogonal frame bundle
\[ \OM=\{p:(\Rr^n,\langle\cdot,\cdot\rangle)\to (T_xM,\eta_x) ~|~\textrm{linear isometry}\}, \]
is a $\O_n(\Rr)$-structure with connection $\omega$ (the Levi-Civita connection). The germs of infinitesimal symmetries of $(\OM,\theta,\omega)$ at a point $u\in \OM$ can be identified with the space of germs of Killing vector fields of $\eta$ at the point $x=\pi(p)$ :
\[ \X(\OM,\theta,\omega)_{p}=\{X\in\X(M):\Lie_X\eta=0 \textrm{ on some open $U\subset M$ containing $\pi(p)$}\}. \]
If this $\O_n(\Rr)$-structure is fully regular then, by Proposition \ref{prop:inf:sym:coframe}, we must have for any point  $p\in \OM$
\[ \dim \X(\OM,\theta,\omega)_{p}=\dim \OM-\dim X_{(\theta,\omega)}. \]
It is easy to construct Riemannian manifolds $(M,\eta)$ for which there are points $p_1,p_2\in \OM$ such that:
\[ \dim \X(\OM,\theta,\omega)_{p_1}\not= \dim\X(\OM,\theta,\omega)_{p_2}. \]

For example, let $M=\Rr$ with coordinate $t$, and consider the Riemannian structure
\[
\eta_t(v,w)=
\begin{cases}
vw & \text{for } t\le 0,\\
(1+f(t))vw & \text{for } t> 0,
\end{cases}
\]
where $f$ is a smooth function such that $f'(t)>0$ if $t>0$, and $f$ and all its derivatives vanish at $t=0$. Then one finds that a (local) Killing vector field for this metric must vanish for $t>0$ and is constant for $t<0$. It follows that:
\[ \dim \X(\OM,\theta,\omega)_p=
\begin{cases}
1 & \text{if } \pi(p)<0,\\
0 & \text{if } \pi(p)> 0.
\end{cases}
\]

\subsection{Generalized space forms}
\label{ex:space:form}
Let $(M,\eta)$ be a pseudo-Riemannian manifold of signature $(p,n-p)$. The associated $O(p,n-p)$-frame bundle
\[ \OM=\{p:(\Rr^n,\langle\cdot,\cdot\rangle_{(p,n-p)})\to (T_xM,\eta_x) ~|~\textrm{linear isometry}\}, \]
is a $O(p,n-p)$ structure with connection (the Levi-Civita connection). As we saw in the previous example, this need not be fully regular. We recall the following result -- see, e.g., \cite{KN1,KN2} for the Riemannnian case and \cite[Section 4.5]{Kriele99} for the pseudo-Riemannian case:

\begin{prop}
Let $(M,\eta)$ be a pseudo-Riemannian manifold of constant scalar curvature. Then for any pair of orthonormal frames $p_1,p_2\in \OM$ there is a local isometry $\varphi$ of $(M,\eta)$ with $\widetilde{\varphi}(p_1)=p_2$.
\end{prop}

The (local) equivalences of $(\OM,\theta,\omega)$ are the lifts of (local) diffeomorphisms $\varphi:M\to M$ which preserve $\eta$. Hence, for a space of constant scalar curvature the $O(p,n-p)$-structure $(\OM,\theta,\omega)$ is fully regular of rank $0$. In particular, $X_{(\theta,\omega)}$ reduces to a point and the classifying algebroid $A_{(\theta,\omega)}$ is actually a Lie algebra. 

If $(M,\eta)$ has constant curvature $\kappa$ its Riemannian curvature tensor takes the form:
\[ R(X,Y) Z=\kappa \Big(\eta(Z,Y)X-\eta(Z,X)Y\Big). \]
Since the torsion vanishes identically, we conclude from Proposition \ref{prop:classify:algbrd} that the classifying algebroid of $(\OM,\theta,\omega)$ is the Lie algebra $A_{(\theta,\omega)}=\Rr^n\oplus\mathfrak{so}(p,n-p)$ with Lie bracket:
\begin{equation}
\label{eq;bracket:const:curv} 
[(u,\al), (v, \be)] = (\al\cdot v - \be\cdot u, [\al,\be] - R(u,v)),
\end{equation}
where $R:\wedge^2\Rr^n\to\mathfrak{so}(p,n-p)$ is given by:
\[ R(u,v)(w)=\kappa \left(\langle w,v\rangle_{(p,n-p)} u-\langle w,u\rangle_{(p,n-p)} v\right).\]

Classically, spaces $(M,\eta)$ with constant scalar curvature are called \emph{space forms}. More generally, a $G$-structure with connection $(P,\theta,\omega)$ may be called a \emph{generalized space form} if it is fully regular of zero constant rank and is torsionless. Equivalently, the corresponding classifying space reduces to a point $X_{(\theta,\omega)}$, i.e., the classifying Lie algebroid is a Lie algebra $A_{(\theta,\omega)}=\Rr^n\oplus\gg$ with Lie bracket:
\[ [(u,\al), (v, \be)] = (\al\cdot v - \be\cdot u, [\al,\be] - R(u,v)),\]
where $R:\wedge^2\Rr^n\to\gg$. In this generality, $R$ is not characterized solo by a number, and one cannot talk about scalar curvature.

%
%
%
%
%
%

\subsection{Homogeneous $G$-structures}
\emph{Homogeneous $G$-structures with connections} are $G$-structures with connection which have large symmetry groups. For a nice overview and references to standard results see \cite{Ballmann00}.


\begin{definition}
A $G$-structure with connection $(P,\theta, \omega)$ is \textbf{(locally) homogeneous} if for any $m_1,m_2\in M=P/G$ there exists a (local) equivalence mapping $m_1$ to $m_2$.
\end{definition}

We note that this condition can be expressed at the level of frames by saying that for any $p,q \in P$ there exists a (local) equivalence (between $G$-saturated open neighborhoods of $p$ and $q$) mapping the orbit of $p$ to the orbit of $q$. For example, the generalized space forms in the previous paragraph are examples of \emph{locally homogeneous} $G$-structures with connection. They become \emph{homogeneous} $G$-structures if, e.g., they are geodesically complete and 1-connected.

%

\begin{theorem}\label{thm:homogeneous}
A $G$-structure with connection $(P,\theta,\omega)$ is locally homogeneous if and only if it is fully regular and the $G$-action on its classifying manifold $X_{(\theta,\omega)}$ is transitive.
\end{theorem}

\begin{proof}
We begin by showing that if $(P,\theta,\omega)$ is locally homogeneous, then it is fully regular. We must show that the dimension of
\[ \mathcal{I}_p = \{ \d_p I: I\in \Inv(P,\theta,\omega)\}\subset T_p^*P \]
does not depend on $p$. 

Let $p, q \in P$, and let $\Phi: U_p \to U_q$ be a local equivalence between $G$-saturated neighbourhoods of $p$ and $q$ mapping the $G$-orbit of $p$ to that of $q$. Then there exists $g \in G$ such that $\Phi(p) = qg$ and $\Phi^*: \mathcal{I}_{qg} \to \mathcal{I}_p$ is an isomorphism. Hence, it is enough to show that $\mathcal{I}_{qg} \simeq \mathcal{I}_q$. 

Note that if  $I \in \Inv(P,\theta,\omega)$, and $g \in G$, then
\[g^*I: P \to \Rr, \quad g^*I(p) = I(pg)\]
is also an invariant function. In fact, for any local equivalence $\Phi: P \to P$ one has that
\[g^*I\Phi(p) = I(\Phi(p)g) = I(\Phi(pg)) = I(pg) = g^*I(p).\]
It then follows that 
\[g^*: \mathcal{I}_{qg} \to \mathcal{I}_q, \quad \d_{qg}I \mapsto \d_q g^*I\]
is an isomorphism. 

Next we show that a fully regular $G$-structure with connection $(P,\theta,\omega)$ is locally homogeneous if and only if the $G$-action on the classifying manifold  $X_{(\theta,\omega)}$ is transitive. Recall that the classifying manifold is $X_{(\theta,\omega)} = P/\sim$ where $p \sim q$ if and only if there exists a local equivalence $\Phi: U_p \to U_q$ such that $\Phi(p) = q$. The natural $G$-action on $X_{(\theta,\omega)}$ is given by $[p]g = [pg]$. It is then clear that this action is transitive if and only if for each frame $q \in P$, the equivalence class of $q$ under local equivalences intersects the $G$-orbit of any frame $p \in P$.
\end{proof}

If $(P,\theta,\omega)$ be a locally homogeneous $G$-structure with connection, then the Lie algebras $\X(P,\theta,\omega)_p$ of germs of infinitesimal symmetries are all isomorphic and 
\[  \dim \X(P,\theta,\omega)_p \geq \dim(P/G). \]
Recall that $\X(P,\theta,\omega)_p$ can also be identified with the kernel of the anchor of the classifying algebroid at $h(p)$. On the other hand, it follows from the Theorem \ref{thm:homogeneous} that the classifying manifold of a locally homogeneous $G$-structure with connection is a homogeneous space
\[ X_{(\theta,\omega)} \simeq G/G_x,\] 
where $x = [p] \in X_{(\theta,\omega)}$, and $G_x$ is the isotropy subgroup at $x$. The isotropy group $G_x$ has geometric meaning: if $m=\pi(p) \in M=P/G$, then $G_x$ is isomorphic to the group of germs at $m$ of diffeomorphisms which fix $m$ and are local symmetries of $(P,\theta,\omega)$. This follows directly from \cite[Thm 7.2]{FS19}.

\smallskip

Properties of the classifying Lie algebroid of a locally homogeneous $G$-structure with connection are reflected in the geometry of the $G$-structure. We will not give a detail discussion here, but the following result is an instance of this relationship:

\begin{theorem}\label{thm:symmetric}
Let $(P,\theta,\omega)$ be a locally symmetric $G$-structure with connection. Then its classifying Lie algebroid satisfies the following properties:
\[ T\equiv 0 \text{ and } F\equiv 0. \]
Conversely, if these conditions hold then the linear connection $\nabla$ on $M=P/G$ induced by $\omega$ is locally symmetric.
\end{theorem}

Before we turn to the proof we recall that a manifold $M$ with a linear connection $\nabla$ is called a (locally) symmetric space if the geodesic symmetries through each $m\in M$ are (local) affine transformations of $(M,\nabla)$. Given a $G$-structure with connection $(P,\theta,\omega)$ we will say that it is (locally) symmetric if the geodesic symmetries are (local) symmetries of the $G$-structure. If a $G$-structure with connection $(P,\theta,\omega)$ is (locally) symmetric then $M=P/G$ with the induced linear connection $\nabla$ is (locally) symmetric. The converse does need to hold: for example, for the usual round metric on the 3-sphere the Levi-Civita connection is symmetric and the associated $O(3)$-structure with connection is also symmetric. However, the associated smaller $SO(3)$-structure with connection is not symmetric (the geodesic symmetries do not preserve orientations of the frames).

Recall also that $(M,\nabla)$ is locally symmetric if and only if its torsion vanishes and its curvature tensor is $\nabla$-parallel -- see, example, \cite{KN2}. It follows that if a $G$-structure with connection $(P,\theta,\omega)$ is locally symmetric then its torsion and curvature maps 
\[\widetilde{T}: P \to \hom(\wedge^2\Rr^n,\Rr^n), \text{ and } \widetilde{R}: P \to \hom(\wedge^2\Rr^n,\gg),\] 
satisfy $\widetilde{T} \equiv 0$ and $\frac{\partial \widetilde{R}}{\partial \theta} \equiv 0$ (see Remark \ref{rmk:covariant:derivative}). 

\begin{remark}\label{rm:-id}
If $(P,\theta,\omega)$ is a locally symmetric $G$-structure, then $-\mathbb{I}\in \GL_n(\Rr)$ must be an element of $G$. In fact, let $S_m$ be the local geodesic symmetry around $m = \pi(p) \in M=P/G$. The lift $\widetilde{S_m}:\B(M)\to \B(M)$ of $S_m$ to the frame bundle maps $p$ to $-p$. Since $S_m$ is a symmetry of the $G$-structure, it follows that $-p \in P$ whenever $p\in P$. It follows also that any invariant function $\Inv(P,\theta,\omega)$ must be invariant under the action of $-\mathbb{I}$. This in turn implies that all elements of $X_{(\theta,\omega)}$ are fixed by $-\mathbb{I}\in G$.
\end{remark}

In order to prove Theorem \ref{thm:symmetric}, we start by proving the following lemma:

\begin{lemma}\label{lemma:covariant:derivative}
Let $(P, \theta, \omega)$ be a locally symmetric $G$-structure with connection,. Then $\frac{\partial I}{\partial \theta} = 0$ for all invariant functions $I \in \mathrm{Inv}(P,\theta, \omega)$.
\end{lemma}

\begin{proof}
Fix $m = \pi(p) \in P/G$. Since any $I\in \mathrm{Inv}(P,\theta,\omega)$ is invariant under both $\widetilde{S_m}$ and $-\mathbb{I}$, it follows that $I$ is also invariant under the map:
\[ \varphi_m:P\to P, \quad \varphi_m = \widetilde{S_m}\circ (-\mathbb{I}).\]
Using the $G$-equivariance of $\theta$ and $\omega$ we see that this map satisfies:
\begin{align*}
\varphi_m^*\theta & = (-\mathbb{I})^*\theta = (-\mathbb{I})\cdot\theta=-\theta,\\\
\varphi_m^*\omega & = (-\mathbb{I})^*\omega = \Ad_{-\mathbb{I}}\omega = \omega.
\end{align*}
In particular, $\varphi_m$ fixes the point $p$ and satisfies:
\[ \d_p\varphi_m(\xi_u)=-\xi_u,\]
where $\xi_u \in \ker \omega_p$ is the unique vector such that $\theta(\xi_u) = u$.

The definition of the coframe derivative of $f \in \mathrm{C}^\infty(P)$ with respect to $\theta$ shows that it can be computed as
\[\frac{\partial f}{\partial \theta}(p)(u) = \d_p f(\xi_u). \]
Hence, it follows that for any $I \in \mathrm{Inv}(P,\theta,\omega)$, we have
\[\frac{\partial I}{\partial \theta}(p)(u) = \d_p I(\xi_u) = \d_p (I \circ \varphi_m)(\xi_u) = \d_pI(-\xi_{u}) = -\frac{\partial I}{\partial \theta}(p)(u). \]
This implies that $\frac{\partial I}{\partial \theta}(p) = 0$. Since $p$ is arbitrary we obtain that $\frac{\partial I}{\partial \theta}\equiv 0$.
\end{proof}

\begin{proof}[Proof of Theorem \ref{thm:symmetric}]
Let $(P,\theta, \omega)$ be a $G$-structure with connection. We denote by 
\[ \widetilde{T}: P \to \hom(\wedge^2\Rr^n,\Rr^n), \qquad \widetilde{R}: P \to \hom(\wedge^2\Rr^n,\gg)\]
the invariant maps corresponding to the structure maps $T: X_{(\theta,\omega)} \to \hom(\wedge^2\Rr^n,\Rr^n)$ and $R: X_{(\theta,\omega)} \to \hom(\wedge^2\Rr^n,\gg)$ of the classifying Lie algebroid $A_{(\theta,\omega)} \to X_{(\theta,\omega)}$. 

If $(P,\theta, \omega)$ is locally symmetric, then we know that $\widetilde{T} \equiv 0$ and so it follows that $T \equiv 0$.  On the other hand, recall that $h^*\mathrm{C}^\infty(X_{(\theta,\omega)}) = \mathrm{Inv}(P,\theta,\omega)$, and for any $f \in \mathrm{C}^\infty(X_{(\theta,\omega)})$, $u\in \Rr^n$, \eqref{eq:Rn:coframe:derivative} gives
\[\d_xf(F_x(u)) = \frac{\partial (h^*f)}{\partial \theta}(p)(u),\]
where $h(p) = x$. It then follows from Lemma \ref{lemma:covariant:derivative} that $F \equiv 0$.

For the converse, it is clear that if $T \equiv 0$ then $\widetilde{T}\equiv 0$, so $\nabla$ is torsion free. Also, using again \eqref{eq:Rn:coframe:derivative}, if $F\equiv 0$ we have
\[ \frac{\partial \widetilde{R}}{\partial \theta}(p)(u) = \Lie_{F(h(p),u)}R\big{|}_{h(p)}= 0. \]
Since $F \equiv 0$ it follows that $\frac{\partial \widetilde{R}}{\partial \theta} \equiv 0$. Equivalently, by Remark \ref{rmk:covariant:derivative}, the curvature $\nabla$-parallel, so $\nabla$ is locally symmetric.
\end{proof}

\begin{remark}
Theorem \ref{thm:symmetric} and Remark \ref{rm:-id} reveal traces of a relationship between the (extremely rich) geometry of symmetric spaces and a special class of of $G$-structure algebroids (see Section \ref{sec:G:structure:algbrd:grpd}), namely those satisfying:
\begin{enumerate}[(i)]
\item $-\mathbb{I} \in G$ and $-\mathbb{I}$ acts trivially on $X$;
\item $F \equiv 0$.
\end{enumerate}
This allows, e.g., to obtain insight into the question of how far a homogeneous $G$-structure is from being symmetric. For instance, it is possible to relate the index of symmetry introduced in \cite{BO17} with the rank of $F$. We will not explore here this relationship and leave it for future work.
\end{remark}

\subsection{Riemannian case}
An interesting class of examples of locally homogeneous $G$-structures with connection is obtain by considering \emph{left (right) invariant $G$-structures with connections} on Lie groups.

\begin{definition}
Let $H$ be a Lie group. A $G$-structure with connection $(\B_G(H), \omega)$ on $H$ is \textbf{left invariant} if the left translation $L_h: H \to H$ is a symmetry of $(\B_G(H), \omega)$, for all $h \in H$.
\end{definition}


Every left invariant $G$-structure with connection on a Lie group $H$ is trivializable, since one can identify $\B_G(H)$ with $H \times G$. Under this identification the tautological and connection forms can be decomposed in a simple way. This allows one to give a more explicit description of the structure equations and therefore of the classifying algebroid of such a structure. We will not explore this in full generality here, but we will focus in the case of $O(n)$-structures, i.e., metrics. We will present bellow an explicit computation for a specific left invariant metric on the Lie group $\mathrm{SU}_2$, which already captures the general case.

Left invariant invariant metric on Lie groups have been extensively studied since the pioneer work of Milnor and others in the 1970's (\cite{AZ79,Gordon80,Milnor76,OT76}). In the language of $G$-structures, these correspond to left invariant $\O_n(\Rr)$-structures with the Levi-Civita connection. Every left invariant metric on a Lie group $H$ is locally homogeneous in the sense of the previous paragraph. It follows that the corresponding $\O_n(\Rr)$-structures with the Levi-Civita connection is fully regular, and moreover that its classifying manifold $X_{(\theta,\omega)}$ is identified with $\O_n(\Rr)/K$, where $K$ is the group of germs of isometries of $H$ at the identity $e\in H$. Note that when forming the quotient $\O_n(\Rr)/K$, we identify $K$ with a subgroup of $\O_n(\Rr)$ as follows: fix a frame $p \in \B_{\O_n(\Rr)}(H)$ over $e$ and let $\varphi \in K$ be a germ at $e$. Then there exists a unique $g \in \O_n(\Rr)$ such that 
\[\d_e\varphi \circ p = pg.\] 
This determines an embedding of $K$ in $\O_n(\Rr)$ as a closed subgroup.

The left invariant metrics on 1-connected 3-dimensional unimodular Lie groups were classified in \cite{HL12}. In what follows we compute the classifying algebroid of a specific left invariant metric on $\mathrm{SU}_2$. The arguments presented below can be generalized to obtain the classifying algebroid of any left invariant metric on a  1-connected and compact Lie group.  

We denote by $\eta$ the left invariant metric on $H=\mathrm{SU}_2$ for which the matrices in 
\[p_0=\left\{\left(\begin{array}{ccc}
0& 0& 1 \\
0 & 0 & 0\\
-1 & 0 & 0  \end{array} \right),
\left(\begin{array}{ccc}
0& 2& 0 \\
-2 & 0 & 0\\
0 & 0 & 0  \end{array} \right), 
\left(\begin{array}{ccc}
0 & 0& 0 \\
0 & 0 & 4\\
0 & -4 & 0  \end{array} \right)\right\},\]
form an orthonormal frame of $T_eH = \mathfrak{so}_3$. A straightforward computation leads to the value of the curvature map at this frame: $R(p_0) \in \hom(\wedge^2\Rr^3, \mathfrak{so}_3)$. If one denote by $e_1, e_2, e_3$ the canonical basis of $\Rr^3$, then one finds:
\begin{align*}
R(p_0)(e_1,e_2) = \left(\begin{array}{ccc}
0& c_1& 0 \\
-c_1 & 0 & 0\\
0 & 0 & 0  \end{array} \right),& \quad R(p_0)(e_1,e_3) = \left(\begin{array}{ccc}
0& 0& c_2 \\
0 & 0 & 0\\
-c_2 & 0 & 0  \end{array} \right),\\
R(p_0)(e_2,e_3) &= \left(\begin{array}{ccc}
0 & 0& 0 \\
0 & 0 & -c_3\\
0 & c_3 & 0  \end{array} \right),
\end{align*}
where 
\[c_1 = \frac{181}{16}, \quad c_2 = \frac{313}{16}, \quad c_3= \frac{599}{16}.\]
The value of the curvature map at any other frame can be obtained using the $G$-equivariance
\[R(pg)(u,v) = \mathrm{Ad}_{g^{-1}}R(p)(gu,gv).\]

In order to determine the classifying manifold $X_{(\theta,\omega)}$ of the orthogonal frame bundle of $(H,\eta)$ we must find the group $K$ of germs of isometries which fixes $e \in H$. Note however, that since $H$ is compact and simply connected, its frame bundle will be a strongly complete realization of its classifying algebroid -- see Definition \ref{def:SC} and Section 6 of \cite{FS19} for details. It then follows that any local isometry of $H$ extends to a global isometry and therefore we can identify $K$ with the group of global isometries of $H$ which fix the identity $e$.  The group $K$ was computed in \cite{HL12} and its is isomorphic to $\Zz_2 \times \Zz_2$, namely it is formed by the matrices:
\[\left(\begin{array}{ccc}
1& 0& 0 \\
0 & 1 & 0\\
0 & 0 & 1  \end{array} \right), \quad \left(\begin{array}{ccc}
1& 0& 0 \\
0 & -1 & 0\\
0 & 0 & -1  \end{array} \right), \quad \left(\begin{array}{ccc}
-1& 0& 0 \\
0 & 1 & 0\\
0 & 0 & -1  \end{array} \right),\quad \left(\begin{array}{ccc}
-1& 0& 0 \\
0 & -1 & 0\\
0 & 0 & 1  \end{array} \right).\]
Therefore, the classifying manifold $X_{(\theta,\omega)}$ of the left invariant metric $\eta$ on $H$ is identified with $\mathrm{SO}_3/(\Zz_2 \times \Zz_2)$, and $R$ can be viewed as a map $X_{(\theta,\omega)} \to \hom(\wedge^2\Rr^3,\mathfrak{so}_3)$.

The classifying Lie algebroid of $(\mathrm{SU}_2, \eta)$ is the trivial vector bundle $A_{(\theta,\omega)} = X_{(\theta,\omega)} \times \Rr^3\oplus \mathfrak{so}_3 \to X_{(\theta,\omega)}$, with Lie bracket on constant sections given by
\[[(u,\al),(v,\be)](x) = (\al\cdot v - \be\cdot u, [\al,\be] - R(x)(u,v)).\] 
Finally its anchor is $\rho(u,\al)_x = F(x)(u) + \psi(\al)_x$ where $\psi: \gg \to \X(X_{(\theta,\omega)})$ denotes the infinitesimal $G$-action on $X_{(\theta,\omega)}$, and and $F(x): \Rr^3 \to T_xX_{(\theta,\omega)}$ Note however that since $\psi_x: \gg \to T_xX_{(\theta,\omega)}$ is injective, it follows that $F(x) = 0$ for all $x \in X_{(\theta,\omega)}$.

\section{$G$-structure groupoids and $G$-structure algebroids}
\label{sec:G:structure:algbrd:grpd}

The Lie algebroids that classify regular $G$-structures with connection form a special class of Lie algebroids. We will now discuss this class of Lie algebroids, called \emph{$G$-structure Lie algebroids} (with connection), as well as their global counterparts, \emph{$G$-structure Lie groupoids} (with connection). As we will will see, a $G$-structure groupoid is a Lie groupoid whose source fibers are $G$-structures hence they yield the appropriate framework to study families of $G$-structures and, in particular, to deal with Cartan's realization problem.

\subsection{$G$-structure groupoids}
We denote by $\G\tto X$ a Lie groupoid with space of arrows $\G$ and space of objects $X$. We use the letters $\s$ and $\t$ for the source and target maps, the symbol $1_x$ for the identity arrow at $x\in X$ and $\gamma_1\cdot\gamma_2$ for the product of the composable arrows $(\gamma_1,\gamma_2)\in \G^{(2)}:=\G\timesst\G$. Also, we denote by $T^{\s} \G=\ker\d\s$ the tangent distribution to the source fibers. We will denote by $\Omega^k_R(\G,V)$ the space of right-invariant $k$-forms on $\G$ with values in $V$. By definition, these are $\s$-foliated forms with values in $V$, i.e., vector bundle maps $\wedge^k T^\s\G\to \G\times V$. The de Rham differential restricts to a differential $\d:\Omega^k_R(\G,V)\to \Omega^{k+1}_R(\G,V)$.

An action of a Lie group $G$ on a Lie groupoid $\G\tto X$ will be called a {\bf $G$-principal action} if:
\begin{enumerate}[(i)]
\item The action is locally free, effective and proper;
\item The source map $\s$ is $G$-invariant
\item The action is compatible with the groupoid multiplication:
\begin{equation}
\label{eq:principal:action} 
(\gamma_1\cdot \gamma_2)\, g= (\gamma_1\, g)\cdot \gamma_2,\quad \forall (\gamma_1,\gamma_2)\in\G^{(2)},\ g\in G.
\end{equation}
\end{enumerate}
In this case we also call $\G$ a {\bf $G$-principal groupoid}. By a {\bf morphism of $G$-principal groupoids} we mean a  $G$-equivariant groupoid morphism $\Phi:\G_1\to\G_2$ between $G$-principal groupoids.

For a $G$-principal groupoid $\G\tto X$ each source fiber $\s^{-1}(x)$ is a $G$-principal bundle over the orbifold $M=\s^{-1}(x)/G$. So a $G$-principal groupoid $\G\tto X$ is a family of $G$-principal bundles parameterized by $X$.

We are mostly interested in $G$-structures. Recalling the characterization of such structures given by Theorem \ref{thm:tensorial-embedding}, one is lead to the following:

\begin{definition}
Given a closed subgroup $G\subset \GL(n,\Rr)$, a {\bf $G$-structure groupoid} consists of a $G$-principal groupoid $\G\tto X$ equipped with a pointwise surjective 1-form $\Theta\in\Omega^1_R(\G,\Rr^n)$ satisfying:
\begin{enumerate}[(i)]
\item $\Theta$ is strongly horizontal: $\Theta_\gamma(v)=0$ iff $v=\tilde{\al}|_\gamma$, for some $\al\in\gg$;
\item $\Theta$ is $G$-equivariant: $g^*\Theta=g^{-1}\cdot\Theta$, for all $g\in G$.
\end{enumerate}
We call $\Theta$ the {\bf tautological form} of the $G$-structure groupoid.

A {\bf morphism} of $G$-structure groupoids $\Phi:\G_1\to\G_2$ is a morphism of $G$-principal groupoids which preserves the tautological forms: $\Phi^*\Theta_2=\Theta_1$.
\end{definition}

Therefore, each source fiber $\s^{-1}(x)$ of a $G$-structure groupoid is a $G$-structure over the orbifold $M=\s^{-1}(x)/G$ with tautological form the restriction $\Theta|_{\s^{-1}(x)}$.

\begin{example}
\label{ex:principal:bundle:principal:grpd}
A $G$-principal action on a manifold $P$ is the same thing as a $G$-principal action on the pair groupoid $\G=P\times P\tto P$. The two actions are related by:
\[  \G\times G\to \G,\quad (p_1,p_2)\, g:=(p_1g,p_2). \]
Condition \eqref{eq:principal:action} holds since we have:
\[ ((p_1,p_2)\cdot (p_2,p_3))\, g=(p_1,p_3)\, g=(p_1\, g,p_3)=((p_1,p_2)\, g)\cdot (p_2,p_3). \]

Similarly, a $G$-structure groupoid on the pair groupoid $P\times P\tto P$ is the same thing as an ordinary $G$-structure on $P\to P/G=M$. The tautological 1-forms $\theta\in\Omega^1(P,\Rr^n)$ and $\Theta\in\Omega^1_R(\G,\Rr^n)$ are related by:
\[ \Theta_{(p_1,p_2)}(v,0)=\theta_{p_1}(v). \]

Given a $G$-structure groupoid $\G\tto X$, each $\s$-fiber $\s^{-1}(x)\to \s^{-1}(x)/G$ has a $G$-structure, and we have a morphism of $G$-structure groupoids covering the target:
\[ 
\vcenter{
\xymatrix{\s^{-1}(x)\times \s^{-1}(x) \ar@<0.2pc>[d] \ar@<-0.2pc>[d]\ar[r] & \G\ar@<0.2pc>[d] \ar@<-0.2pc>[d]\\
\s^{-1}(x) \ar[r]_{\t} & X}}
\qquad (\gamma_1,\gamma_2)\mapsto \gamma_1\cdot \gamma_2^{-1}.
\]
\end{example}

There is a slightly different point of view on $G$-principal and $G$-structure groupoids, which will be useful when we introduce their infinitesimal versions.

First, given a $G$-principal groupoid $\G\tto X$ we have a $G$-action on $X$ defined by:
\begin{equation}
\label{eq:X:action}
X\times G\to X,\quad x\, g:=\t(1_x\, g).
\end{equation}
For this action, $\s:\G\to X$ is $G$-invariant and $\t:\G\to X$ is $G$-equivariant. 
 
Next, recall that given a (right) $G$-action on a manifold $X$ one can form the action Lie groupoid $X\rtimes G\tto X$: an arrow is a pair $(x,g)$ with source $x$ and target $xg$, and composition of arrows is given by:
\[ (y,h)\cdot (x,g)=(x,gh),\quad \text{if }y=xg. \]

Finally, one can define the groupoid morphism:
\begin{equation}
\label{eq:action:morphism}
\iota: X\rtimes G\to \G,\quad (x,g)\mapsto 1_x\, g.
\end{equation}
This will be called the \textbf{action morphism} of the the $G$-principal groupoid.

To state the next result, let us introduce the following notions. A groupoid morphism $\iota: X\rtimes G\to \G$ is called
\begin{enumerate}[(a)]
\item {\bf effective} if given $e\not=g\in G$ there exists an $x\in X$ with $\iota(x,g)\not=1_x$;
\item {\bf injective} if for any $x\in X$ and $g\in G$ one has $\iota(x,g)=1_x$ if and only if $g=e$;
\item {\bf locally injective} if for any $x\in X$ there is an open $e\in U_x\subset G$ such that $g\in U_x$ and $\iota(x,g)=1_x$ if and only if $g=e$.
\end{enumerate}

\begin{prop}
\label{prop:principal:G:action}
Let $\G\tto X$ be a Lie groupoid. If $\G\times G\to \G$ is a $G$-principal action, then:
\begin{enumerate}[(i)]
\item \eqref{eq:X:action} defines a $G$-action on $X$; 
\item \eqref{eq:action:morphism} defines an effective, locally injective, Lie groupoid morphism;
\item the action takes the form:
\begin{equation}
\label{eq:principal:action:formula}
\G\times G\to \G,\quad \gamma\, g:= \iota(\t(\gamma),g)\cdot \gamma.
\end{equation}
\end{enumerate}
Conversely, given an action $X\times G\to X$ and an effective, locally injective, Lie groupoid morphism $\iota: X\rtimes G\to \G$, \eqref{eq:principal:action:formula} defines a $G$-principal action on $\G$.
\end{prop}

For a proof of this proposition we refer to \cite{FS19}. It shows that we can define the $G$-principal action on $\G$ by specifying first a $G$-action on $X$ and then an effective, locally injective, groupoid morphism $\iota: X\rtimes G\to \G$. We will often use this alternative point of view.  From this perspective, a morphism of $G$-principal groupoids can be characterized as a morphism of groupoids
\[
\xymatrix@R=20pt{
\G_1\ar[r]^{\Phi}\ar@<0.2pc>[d] \ar@<-0.2pc>[d]  & \G_2\ar@<0.2pc>[d] \ar@<-0.2pc>[d]  \\
X_1\ar[r]_{\phi} & X_2
}
\]
which intertwines the actions morphisms: 
\[ \Phi\circ \iota_1=\iota_2\circ (\phi\times I). \]

The action morphism $\iota: X\rtimes G\to \G$ also allows us to define another action of $G$ on $\Gamma$, namely the {\bf action by inner automorphisms}:
\begin{equation}
\label{eq:inner:G:action}
\G\times G\to\G, \quad \gamma \odot g:=\iota(\t(\gamma),g)\cdot \gamma \cdot \iota(\s(\gamma), g)^{-1}.
\end{equation}
In general, the inner action of $G$ on $\Gamma$ \emph{does not determine} the original $G$-action on $\Gamma$. For example, the action morphism maybe non-trivial, while the inner action could be trivial. 

\subsection{Connections on $G$-structure groupoids}

The notion of connection on a $G$-principal groupoid $\G\tto X$, much like principal bundle connections, can be defined either as an invariant distribution or as a $\gg$-valued 1-form. We will choose the later approach and refer to  \cite{FS19} for the former, as well as the equivalence between the two approaches.

\begin{definition}
A \textbf{connection 1-form} on a $G$-principal groupoid $\G\tto X$ is a $\gg$-valued 1-form $\Omega\in\Omega^1_R(\G,\gg)$ satisfying:
\begin{enumerate}[(i)]
\item $\Omega$ is vertical: $\Omega(\tilde{\al})=\al$, for all $\al\in\gg$;
\item $\Omega$ is $G$-equivariant: $g^*\Omega=\Ad_{g^{-1}}\cdot \Omega$, for all $g\in G$.
\end{enumerate}
\end{definition}

The restriction of a connection 1-form $\Omega\in\Omega^1_R(\G,\gg)$ to each fiber $\s^{-1}(x)$ gives an ordinary connection 1-form $\omega\in\Omega^1(\s^{-1}(x),\gg)$. The corresponding horizontal distributions assemble to a  distribution on $\G$ given by:
\[ \H:=\{v\in T^{\s} \G:\Omega(v)=0\}. \]
The {\bf curvature 2-form} of a connection $\Omega$ is a $\gg$-valued 2-form 
\[ \Curv(\Omega)\in \Omega^2_R(\G,\gg), \]
which measures the failure of the horizontal distribution $\H$ in being integrable. It is defined by:
\[ \Curv(\Omega)(v,w):=\d \Omega(h(v),h(w)), \quad (v,w\in T^{\s} \G), \]
where $h:T^{\s} \G\to \H$ denotes the projection.  

The restriction of $\Curv(\Omega)$ to the s-fiber $\s^{-1}(x)$ is the usual curvature 2-form of the induced connection on $\s^{-1}(x)\to \s^{-1}(x)/G$. This leads immediately to the fact that a connection $\Omega$ on a $G$-principal groupoid $\G\tto X$ satisfies:
\begin{description}
\item[1st structure equation]
\[ \d\Omega=-\Omega\wedge \Omega +\Curv(\Omega); \]
\item[1st Bianchi's identity]
\[ \d\Curv(\Omega)|_{\H}=0. \]
\end{description}

Assume now that $\G\tto X$ is a $G$-structure groupoid with connection $\Omega$.  Denoting the tautological form by $\Theta$, we define the {\bf torsion of the connection} to be the right-invariant 2-form $\Tors(\Omega)\in\Omega^2_R(\G,\Rr^n)$ given by:
\[ \Tors(\Omega)(v,w)=\d\Theta(h(v),h(w)), \quad (v,w\in T^{\s} \G). \]
The restriction of $\Tors(\Omega)$ to the source fiber $\s^{-1}(x)$ is the (usual) torsion 2-form of the induced connection on $\s^{-1}(x)\to \s^{-1}(x)/G$ and we find that the following hold:
\begin{description}
\item[2nd structure equation]
\[ \d\Theta=-\Omega\wedge\Theta +\Tors(\Omega); \]
\item[2nd Bianchi's identity]
\[ \d\Tors(\Omega)|_{\H}= - (\Curv(\Omega)\wedge \Theta)|_{\H}. \]
\end{description}

One can also consider morphisms of $G$-structure groupoids with connection. They amount to a morphism of groupoids $\Phi:\G_1\to\G_2$ which is $G$-equivariant and preserves both the tautological and connection 1-forms:
\[ \Phi^*\Theta_2=\Theta_1,\quad \Phi^*\Omega_2=\Omega_1. \]

\subsection{$G$-structure algebroids}

We denote by $\Der(A)$ the space of \textbf{Lie algebroid derivations} of a Lie algebroid $A$. An element $D\in\Der(A)$ is a derivation of the vector bundle $A\to X$ that acts as a derivation of the bracket.  So $D:\Gamma(A)\to\Gamma(A)$ is a linear map for which there exists a vector field $\sigma_D\in\X(M)$ such that:
\[ D(fs)=fDs+\sigma_D(f) s \qquad (s\in\Gamma(A),\, f\in C^\infty(X)), \]
and moreover:
\[ D([s_1,s_2]_A)=[D(s_1),s_2]_A+[s_1,D(s_2)]_A, \qquad (s_1,s_2\in\Gamma(A)). \]
We call $\sigma_D$ the \textbf{symbol} of the derivation. The commutator of derivations turns $\Der(A)$ into a Lie algebra.

Given a Lie algebra $\gg$, by an infinitesimal $\gg$-action on $A$ we mean a Lie algebra map $\widehat{\psi}:\gg\to \Der(A)$. Composing $\widehat{\psi}$ with the symbol map gives an infinitesimal $\gg$-action on the base $X$ of the Lie algebroid $\psi:\gg\to \X(X)$, so that:
\[ \widehat{\psi}(\al)(fs)=f\widehat{\psi}(\al)(s)+\psi(\al)(f) s. \]

Consider a $G$-action on a Lie algeboid $A$:
\[ A\times G\to A,\quad (a,g)\mapsto a\odot g. \]
If the action is by Lie algebroid automorphisms, then it induces an infinitesimal $\gg$-action $\widehat{\psi}:\gg\to \Der(A)$, given by:
\[ \widehat{\psi}(\al)(s):=\left.\frac{\d }{\d t}\right|_{t=0} (\exp(t\al))^* s, \quad \al\in\gg,\, s\in\Gamma(A). \]
A $G$-action on a Lie algebroid $A$ is called a {\bf $G$-principal action} if
\begin{enumerate}[(i)]
\item The action is by Lie algebroid automorphisms;
\item  The infinitesimal action $\widehat{\psi}:\gg\to\Der(A)$ satisfies:
\[ \widehat{\psi}(\al)=[i(\al),-]. \]
where $i:X\rtimes \gg\to A$ is an injective algebroid morphism. 
\end{enumerate}
In this case, we call $A$ a {\bf $G$-principal algebroid} and $i:X\rtimes \gg \to A$ the {\bf action morphism}.

Notice that the action morphism $i:X\rtimes \gg \to A$ is part of the data defining a $G$-principal Lie algebroid: there could be more that one such morphism determining the same infinitesimal action $\widehat{\psi}:\gg\to\Der(A)$. When $G$ is connected, this morphism determines the $G$-action. 

\begin{definition}
Given a closed subgroup $G\subset \GL(n,\Rr)$, a {\bf $G$-structure algebroid} consists of a $G$-principal algebroid $A\to X$ with action morphism $i:X\rtimes \gg\to A$ equipped with a fiberwise surjective $A$-form $\theta\in\Omega^1(A,\Rr^n)$ satisfying:
\begin{enumerate}[(i)]
\item strong horizontality:
\[ \theta_x(a)=0 \quad\text{iff}\quad a=i(x,\al),\text{ for some }\al\in\gg. \]
\item $G$-equivariance:
\[ \theta_{x\cdot g}(a\odot g)=g^{-1}\cdot\theta_x(a),\quad \forall g\in G. \]
\end{enumerate}
\end{definition}

We call $\theta$ the {\bf tautological form} of the $G$-structure algebroid. The definition of a $G$-structure algebroid $A\to X$ implies that:
\[ \rank A=n+\dim G. \]
There is no restriction on the dimension of the base $X$, which can be arbitrary. 

\begin{example}
\label{ex:principal:bundle:algebroid}
Let $(P,\theta)$ be $G$-structure. Then the lifted $G$-action on the tangent bundle $A=TP\to P$ yields a $G$-structure algebroid with action morphism
\[ i:P\times\gg\to TP, \quad (p,\al)\mapsto \tilde{\al}_p, \]
and $A$-form $\theta$. This is the infinitesimal version of the $G$-structure groupoid of Example \ref{ex:principal:bundle:principal:grpd}.
\end{example}

\begin{example}
\label{ex:classifying:alegbroid}
Our main example is, of course, the classifying Lie algebroid $A_{(\theta,\omega)}$ of fully regular $G$-structure with connection $(P,\theta,\omega)$. We saw in Section \ref{sec:canonical:form} that this Lie algebroid comes with an injective algebroid morphism
\[ i:X_{(\theta,\omega)}\times\gg\to A_{(\theta,\omega)}, (x,\al)\mapsto (x,(0,\al)). \]
and an action $G\times A_{(\theta,\omega)}\to A_{(\theta,\omega)}$ by Lie algebroid automorphisms. One checks that the induced infinitesimal action $\widehat{\psi}:\gg\to\Der(A)$ satisfies:
\[ \widehat{\psi}(\al)=[i(\al),-]. \]
The projection
\[ A_{(\theta,\omega)}=X_{(\theta,\omega)}\times(\Rr^n\oplus\gg)\to \Rr^n \]
defines a form $\theta\in\Omega^1(A_{(\theta,\omega)},\Rr^n)$ which satisfies all the conditions of the definition.
\end{example}

There is a differentiation functor taking $G$-principal groupoids to $G$-principal algebroids. We state here the main results and refer to \cite{FS19} for more details and proofs.

\begin{prop}
\label{prop:principal:G:action:integrable}
Let $\G\tto X$ be a $G$-principal Lie groupoid with action morphism $\iota:X\rtimes G\to \G$. Then its Lie algebroid $A\to X$ is a $G$-principal algebroid with action morphism $i=\iota_*:X\rtimes \gg \to A$. 

Conversely, given a $G$-principal algebroid $A\to X$, a Lie groupoid $\G\tto X$ integrating $A$ admits a (unique) $G$-principal action inducing the $G$-action on $A$ provided the action morphism $i:X\rtimes \gg\to A$ integrates to an effective Lie groupoid morphism:
\[ \iota: X\rtimes G\to \G. \]
\end{prop}

\begin{remark}
\label{rem:correspodence:G:structures}
Assume that a $G$-principal groupoid $\G$ corresponds to a $G$-principal algebroid $A$, as in the previous proposition. Then the fact that right-invariant forms on $\G$ are in bijection with $A$-forms leads immediately to a correspondence of $G$-structures, namely:
\begin{equation}
\label{correspondence:G:structures:grpd:algbrd}
\left\{\txt{$\G$-tautological forms\\ $\Theta\in\Omega^1_R(\G,\Rr^n)$ \\ \,}\right\}
\tilde{\longleftrightarrow}
\left\{\txt{$A$-tautological forms\\ $\theta\in\Omega^1(A,\Rr^n)$\\ \,} \right\}.
\end{equation}
\end{remark}


There is also a natural notion of {\bf morphism of $G$-principal Lie algebroids}: it is a Lie algebroid morphism between $G$-principal Lie algebroids
 \[
 \xymatrix@R=20pt{
 A_1\ar[d]\ar[r]^{\Phi} & A_2\ar[d] \\
 X_1\ar[r]_{\phi} & X_2
 }
 \]
 which is $G$-equivariant and which intertwines the action morphisms:
 \[ \Phi\circ i_1=i_2\circ(\phi\times I). \]
A morphism of $G$-principal groupoids $\Phi:\G_1\to \G_2$ induces a morphism of the associated $G$-principal algebroids $\Phi_*:A_1\to A_2$. The converse, in general, fails unless $\G_1$ is the so called \emph{canonical $G$-integration} of $A_1$ -- see Section \ref{sec:G:integration} and \cite{FS19}.

One defines a {\bf morphism} of $G$-structure algebroids to be a morphism of $G$-principal algebroids $\Phi:A_1\to A_2$ which additionally preserves the tautological forms: $\Phi^*\theta_2=\theta_1$. A morphism of $G$-structure algebroids $\Phi:A_1\to A_2$ is necessarily a fiberwise isomorphism.

\subsection{Connections on $G$-structure algebroids}

Recalling that the right-invariant forms on a Lie groupoid $\Gamma$ correspond to $A$-forms on its Lie algebroid, one is lead to the notion of connection on a $G$-principal algebroid:

\begin{definition}
A \textbf{connection 1-form} on a $G$-principal algebroid $A\to X$ is a $\gg$-valued $A$-form $\omega\in\Omega^1(A,\gg)$ satisfying:
\begin{enumerate}[(i)]
\item It is vertical relative to the morphism $i:X\rtimes\gg\to A$:
\[ \omega(i(x,\al))=\al,\quad \forall\, \al\in\gg,\ x\in M; \]
\item It is $G$-equivariant for the $G$-action on $A$ by automorphisms:
\[ \omega_{xg}(a \odot g)=\Ad_{g^{-1}}\cdot \omega_x(a),\quad \forall\, a\in A_x,\ g\in G.\]
\end{enumerate}
\end{definition}

Notice that a connection 1-form $\omega\in\Omega^1(A,\gg)$ yields a \emph{horizontal sub-bundle} $H=\ker\omega\subset A$ such that:
\[ A=H\oplus \Im i. \]
This sub-bundle is $G$-invariant and determines the connection 1-form uniquely -- see \cite{FS19}.  

\begin{example}
Returning to our main example of the classifying Lie algebroid $A_{(\theta,\omega)}$ of a fully regular $G$-structure with connection $(P,\theta,\omega)$, this is a $G$-structure algebroid -- see Example \ref{ex:classifying:alegbroid} --  with connection form $\omega\in\Omega^1(A_{(\theta,\omega)},\gg)$ the projection:
\[ A_{(\theta,\omega)}=X_{(\theta,\omega)}\times(\Rr^n\oplus\gg)\to \gg. \]
\end{example}

Naturally, one defines the {\bf curvature 2-form} of a connection $\omega$ to be the $A$-form $\Curv(\omega)\in\Omega^2(A,\gg)$ given by:
\[ \Curv(\omega)(a_1,a_2):=\d_A\omega(h(a_1),h(a_2)), \quad a_i\in A, \]
where $h:A\to H$ denotes the projection and $\d_A:\Omega^k(A,V)\to \Omega^{k+1}(A,V)$ the Lie algebroid differential. The connection is flat, i.e., $\Curv(\omega)=0$ if and only if $H\subset A$ is a Lie subalgebroid, and
the following hold:
\begin{description}
\item[1st structure equation]
\[ \d_A\omega=-\omega\wedge\omega +\Curv(\omega); \]
\item[1st Bianchi's identity]
\[ \d_A\Curv(\omega)|_H=0. \]
\end{description}

Assuming now that $A\to X$ is a $G$-structure algebroid with connection 1-form $\omega\in\Omega^1(A,\gg)$, we define the {\bf torsion of the connection} to be the $\Rr^n$-valued $A$-form $\Tors(\omega)\in\Omega^2(A,\Rr^n)$ given by:
\[ \Tors(\omega)(a_1,a_2):=\d_A\theta(h(a_1),h(a_2)), \quad a_i\in A, \]
where $\theta\in\Omega^1(A,\Rr^n)$ denotes the tautological 1-form. The torsion and the curvature satisfy:
\begin{description}
\item[2nd structure equation]
\[ \d_A\theta=-\omega\wedge\theta +\Tors(\omega); \]
\item[2nd Bianchi's identity]
\[ \d_A\Tors(\omega)|_H=-(\Curv(\omega)\wedge \theta)|_H. \]
\end{description}

\begin{remark}
\label{rem:correspodence:connections}
Under the correspondence \eqref{correspondence:G:structures:grpd:algbrd} between $G$-structure groupoids and $G$-structure algebroids, if a connection $\Omega$ on $\G$ corresponds to a connection $\omega$ on $A$, then the associated torsions $\Tors(\Omega)$ and $\Tors(\omega)$ correspond to each other, so that:
\[ \Tors(\omega)_x=\Tors(\Omega)_{1_x}, \quad \forall x\in X. \]
Moreover, the structure equations and the Bianchi identities also correspond to each other.
\end{remark}

Finally, there is also a notion of morphism of $G$-structure algebroids with connection: it is a morphism of algebroids $\Phi:A_1\to A_2$ which is $G$-equivariant, intertwines the action morphisms, and preserves the tautological and connection 1-forms:
\[ \Phi^*\theta_2=\theta_1,\quad \Phi^*\omega_2=\omega_1. \]

\subsection{Canonical form of a $G$-structure algebroid with connection}
\label{sec:canonical:form:2}

Assume that $\pi:A\to X$ is a $G$-structure algebroid with tautological form $\theta$ and connection form $\omega$. We obtain a vector bundle isomorphism to the trivial bundle:
\[ 
A\stackrel{\cong}{\longrightarrow} X\times (\Rr^n\oplus\gg), \quad a\mapsto (\pi(x),\theta(a),\omega(a)). 
\]
We can re-express the torsion and the curvature under this isomorphism as maps:
\begin{itemize}
\item $T: X \to \hom(\wedge^2\Rr^n, \Rr^n)$: $T(x)(v,w)=\Tors(\omega)_x(v,w)$;
\item $R : X\to \hom(\wedge^2\Rr^n,\gg)$: $R(x)(v,w)=\Curv(\omega)_x(v,w)$.
\end{itemize}
Moreover, under this isomorphism:
\begin{itemize}
\item the $G$-action on $A$ takes the form: $(x, u, \al)\odot g = (x\, g, g^{-1}\, u, \Ad_{g^{-1}}\cdot \al)$;
\item the action morphism is the inclusion $i:X\rtimes \gg\to A$, $(x,\al))\mapsto (x,0,\al)$;
\item the tautological form is the projection $\theta:X\times(\Rr^n\oplus\gg)\to \Rr^n$;
\item the connection 1-form is the projection $\omega:X\times(\Rr^n\oplus\gg)\to \gg$.
\end{itemize}
One checks also that the Lie bracket is then given on constant sections by:
\begin{equation}
\label{eq:canonical:Lie:bracket}
[(u,\al), (v, \be)] = (\al\cdot v - \be\cdot u - T(u,v), [\al,\be]_\gg - R(u,v)),
\end{equation}
while the anchor takes the form:
\begin{equation}
\label{eq:canonical:anchor} 
\rho(u,\al)=F(u) + \psi(\al), \quad (u,\al)\in\Rr^n\oplus\gg,
\end{equation}
where $F:X\times\Rr^n\to TX$ is a $G$-equivariant bundle map and $\psi:\gg\to \X(X)$ denotes the infinitesimal $G$-action on $X$. We call this the \textbf{canonical form} of a $G$-structure Lie algebroid with connection.

\begin{remark}
Our usage here of the term \emph{canonical form} is consistent with the usage in Section \ref{sec:canonical:form} for the case of the classifying Lie algebroid of a $G$-structure with connection. One may wonder, if any \emph{transitive} $G$-structure algebroid with connection arises as the classifying Lie algebroid of a $G$-structure with connection. This is a special instance of the \emph{realization problem} to be discussed in the next section.
\end{remark}

\subsection{$G$-integrations}
\label{sec:G:integration}
A fundamental question, suggested by the results above, is the $G$-integrability problem:
\begin{itemize}
\item When does a $G$-structure algebroid with connection arise from a $G$-structure groupoid with connection?
\end{itemize}
A complete solution to this problem is presented in \cite{FS19}. Here we will give a brief overview, focusing on the results most relevant for the theory of (fully regular) $G$-structures.

A first observation is that the connection and tautological form do not play any role here -- see Remarks \ref{rem:correspodence:G:structures} and \ref{rem:correspodence:connections}.  Therefore, we can assume that $A\to X$ is a $G$-principal algebroid and look for a $G$-principal groupoid $\G\tto X$ integrating it. If the latter exists we will say that $A$ is \emph{$G$-integrable} and call $\G$ a \emph{$G$-integration} of $A$.

A Lie algebroid $A\to X$ may fail to be integrable and the obstructions to integrability are well understood (see \cite{CrainicFernandes,CrainicFernandes:lectures}). When $A$ is integrable it may have many integrations, but among the source-connected integrations, there is a unique (up to isomorphism) maximal integration $\Sigma(A)\tto X$.  One can characterize the groupoid $\Sigma(A)$ as the one with 1-connected source fibers.

For a $G$-principal algebroid $A$, it may happen that it is integrable, but not $G$-integrable. If $A$ is $G$-integrable, one can also look for a maximal source connected $G$-integration.  Such integration always exist and this is the version of Lie's 1st theorem for $G$-integrations:

\begin{theorem}[Lie I \cite{FS19}]
\label{thm:canonical:G:integration}
Let $A\to X$ be a $G$-principal algebroid  which is $G$-integrable. Then there exists a unique (up to isomorphism) $G$-principal groupoid $\Sigma_G(A)\tto X$ which is characterized by either of the following:
\begin{enumerate}[(a)]
\item $\Sigma_G(A)$ is an $\s$-connected $G$-integration and the orbifold fundamental groups of $\s^{-1}(x)/G$ are trivial;
\item $\Sigma_G(A)$ is maximal among $\s$-connected $G$-integrations of $A$: for any $\s$-connected $G$-integration $\G$ there exists a unique \'etale, surjective, morphism of $G$-principal groupoid $\Sigma_G(A)\to\G$. 
\end{enumerate}
\end{theorem}

\begin{example} 
\label{ex:canonical:integ:principal:bundle}
Let $\pi: P \to M$ be a connected principal $G$-bundle. We saw in Example \ref{ex:principal:bundle:algebroid} that $TP\to P$ is $G$-principal algebroid.  The canonical $G$-integration of $TP$ can be obtained as follows. Let $q: \widetilde{M} \to M$ be the universal covering space of $M$ and consider the pullback diagram
\[\xymatrix{q^*P \ar[d]_\pi \ar[r]^{\hat{q}} & P\ar[d]^\pi\\
\widetilde{M} \ar[r]_q & M.}\]
On the one hand $q^*P$ is a principal $G$-bundle over $\widetilde{M}$ and therefore caries a right principal $G$-action. On the other hand, $q^*P$ is a $\pi_1(M)$-covering of $P$. The canonical $G$-integration of $TP$ is the gauge groupoid corresponding to the $\pi_1(M)$-principal bundle $\hat{q}: q^*P \to P$:
\[ \Sigma_G(TP)=(q^*P\times q^*P)/\pi_1(M)\tto P, \]
with the $G$-action $[(p_1,p_2)]\, g:=[(p_1g,p_2)]$. This $G$-integration covers the pair groupoid $P\times P\tto P$, which is also a $G$-integration -- see Example \ref{ex:principal:bundle:principal:grpd}.
\end{example}

Using the canonical $G$-integration, there is also a version of Lie's 2nd Theorem:

\begin{theorem}[Lie II \cite{FS19}]
If $\phi: A \to B$ is a morphism between two $G$-integrable $G$-principal algebroids, then there exists a unique morphism between their canonical $G$- integrations:
\[ \Phi: \Sigma_G(A) \to \Sigma_G(B),\qquad \Phi_*=\phi \] 
\end{theorem}

Let us now turn to Lie's 3rd Theorem. First, if a $G$-principal algebroid $A\to X$ is $G$-integrable then it is obviously integrable. The converse, however, does not hold: a $G$-principal algebroid $A\to X$ maybe integrable without being $G$-integrable. For a $G$-principal algebroid $A\to X$ there are certain \textbf{$G$-monodromy groups} $\cN^G_x$ which contain the usual monodromy groups obstructing integrability:
\[ \cN_x\subset \cN^G_x\subset \Ker\rho_x. \]
Then one has:

\begin{theorem}[Lie III \cite{FS19}]
Let $A\to X$ be a $G$-principal algebroid. Then $A$ is $G$-integrable if and only if the $G$-monodromy groups are uniformly discrete, i.e., if and only if there is an open neighborhood $U\subset A$ of the zero section such that:
\[ \cN^G_x\cap U=\{0_x\}, \quad \forall x\in X. \]
\end{theorem}

Once we known that a $G$-principal algebroid $A$ is $G$-integrable, it follows from Theorem \ref{thm:canonical:G:integration} that any $\s$-connected $G$-integration $\G$ of $A$ can be obtained as a quotient of the canonical $G$-integration
\[ \G:=\Sigma_G(A)/\Delta\tto X , \]
where $\Delta\subset \Sigma_G(A)$ is any discrete bundle of subgroups satisfying:
\begin{enumerate}[(a)]
\item $\Delta_x$ is contained in the center of the isotropy group $\Sigma_G(A)_x$;
\item the image of the action morphism $\iota:X\rtimes G\to \Sigma_G(A)$ intersects $\Delta$ only in the identity section.
\end{enumerate}

\begin{example}
\label{ex:curvature:2}
Consider the $\SO_n(\Rr)$-structure Lie algebra $A=\Rr^n\oplus\so_n(\Rr)$, with Lie bracket in canonical form:
\[ [(u,\al), (v, \be)] = (\al\cdot v - \be\cdot u, [\al,\be] - R(u,v)), \]
where:
\[ R(u,v)(w)=\kappa\left(\langle w,v\rangle u-\langle w,u\rangle v\right). \]
As we saw in Section \ref{ex:space:form}, this is the classifying Lie algebroid associated with an oriented Riemannian manifold $(M,\eta)$ with constant scalar curvature $\kappa$. Depending on the value of $\kappa$ one finds that this Lie algebra is isomorphic to
\[
A\simeq
\left\{
\begin{array}{l}
\so_{n+1}(\Rr), \text{ if $\kappa>0$,}\\
\\
\so_n(\Rr)\ltimes\Rr^n, \text{ if $\kappa=0$},
\\
\\
\so_{(n,1)}(\Rr), \text{ if $\kappa<0$.}\\
\end{array}
\right.
\]

Consider, e.g., the case $\kappa>0$. The canonical $\SO_n(\Rr)$-integration is 
\[ \Sigma_{\SO_n}(A)=\SO_{n+1}(\Rr), \] 
since $\s^{-1}(x)/\SO_n(\Rr)\simeq \Ss^n$ is 1-connected. As for other integrations observe that the center of $\SO_{n+1}(\Rr)$ is:
\[
Z(\SO_{n+1}(\Rr))=
\left\{
\begin{array}{l}
\quad \{ I \}, \text{ if $n$ is even,}\\
\\
\{I,-I\}, \text{ if $n$ is odd.}
\end{array}
\right.
\]
So if $n$ is even the only connected $\SO_n(\Rr)$-integration is the canonical one, while if $n$ is odd there is another one, namely
\[ \G=\SO_{n+1}(\Rr)/\{I,-I\}=\mathrm{PSO}(n+1,\Rr). \]
The corresponding manifold of constant positive curvature is of course $\mathbb{RP}^n$. This manifold is oriented since $n$ is even, and this also explains why we don't find it for odd $n$. If instead we consider $A$ as a $\O_n(\Rr)$-algebroid, then we would recover $\mathbb{RP}^n$ for any positive integer $n$.

Observe that when $\kappa=0$ we obtain as canonical $\SO_n(\Rr)$-integration the special Euclidean group
\[ \Sigma_{\SO_n}(A)=\SO_{n}(\Rr)\ltimes \Rr^n. \]
This gives rise of course, to Euclidean space as an oriented manifold of zero curvature. Since the center of this group is trivial so this is the only $\SO_n(\Rr)$-integration. In particular, we cannot obtain the torus with the flat metric by looking only at $\SO_n(\Rr)$-integrations.
\end{example}

\section{$G$-realizations and $G$-structures with connection}
\label{sec:G:realizations}

The link between $G$-structure algebroids and the classifying Lie algebroid of a single $G$-structure is provided by the notion of a $G$-realization. This notion also allows one to understand solutions of Cartan's Realization Problem, which we formulate in this section.

\subsection{$G$-realizations}

We saw that for a fully regular $G$-structure with connection $(P,\theta,\omega)$ we have a Lie algebroid map into the classifying Lie algebroid $TP\to A_{(\theta,\omega)}$. This is abstracted for any $G$-structure algebroid as follows:

\begin{definition}
A {\bf $G$-realization} of a $G$-structure algebroid $A\to X$ (with connection) is a $G$-structure $P\to M$ (with connection) together with a morphism of $G$-structure algebroids (with connection):
\[
\vcenter{\vbox{ 
 \xymatrix@R=20pt{
 TP\ar[d]\ar[r]^H& A\ar[d] \\
 P\ar[r]_h & X
 }}}
 \]
One calls $h:P\to X$ the {\bf classifying map} of the $G$-realization. 
\end{definition}

The $G$-realizations of $A$ form a category. A {\bf morphism of $G$-realizations} of a $G$-structure algebroid $A\to X$ is a map $\Phi:P_1\to P_2$, commuting with the classifying maps $h_i:P_i\to X$ and yielding a commutative diagram of morphisms of $G$-structure algebroids (with connection):
\[ 
\xymatrix{
TP_1\ar[dr]_{H_1}\ar[rr]^{{\d\Phi}} & &TP_2\ar[dl]^{H_2} \\
 & A
}
\]

In order to see how one can construct realizations, let us recall that if $\G\tto X$ is a Lie groupoid with Lie algebroid $A$ one defines its {\bf Maurer-Cartan form} to be the right-invariant, $A$-valued, 1-form $\wmc$ given by:
\[ \wmc(\xi)=\d R_{\gamma^{-1}} \xi, \quad \text{if }\xi\in T^{\s}_\gamma\G. \]
Equivalently, we can view $\wmc$ as a bundle map:
\[
\xymatrix@C=30pt{
T^{\s} \G\ar[d]\ar[r]^{\wmc} & A\ar[d] \\
\G \ar[r]_{\t} & X
}
\]
The $1$-form $\wmc$ satisfies the Maurer-Cartan equation which can be equivalently stated as saying that this bundle map is a Lie algebroid morphism (see, e.g., \cite{FernandesStruchiner1}).

\begin{example}
\label{ex:G:realization:groupoid}
Let $A\to X$  be a $G$-structure algebroid with connection $\omega$ and assume that $\G\tto X$ is a $G$-structure groupoid with connection $\Omega$ integrating it. Then each source fiber $\s^{-1}(x)\to \s^{-1}(x)/G$ is a $G$-structure with connection $\Omega|_{\s^{-1}(x)}$. The restriction of the Maurer-Cartan form gives a morphism of $G$-structure algebroids with connection:
\[
 \xymatrix@R=20pt{
 T(\s^{-1}(x))\ar[d]\ar[r]^---{\wmc} & A\ar[d] \\
 \s^{-1}(x)\ar[r]_{t} & X
 }
 \]
Hence, $(\s^{-1}(x),\wmc|_{\s^{-1}(x)},\t)$ is a $G$-realization of $A$.
\end{example}

We will see that, in general, a fully regular $G$-structure with connection fails to be a source fiber of a $G$-integration of the classifying algebroid. This may fail even if the  classifying algebroid is $G$-integrable. Eventually, we will give a characterization of those fully regular $G$-structures that arise as $\s$-fibers. 

Note that in the case where the $G$-realization consists of a source fiber $\s^{-1}(x)$ of a Lie groupoid integrating $A$, its classifying map (the target map) is a submersion onto the leaf containing $x$. For a fully regular $G$-structure the classifying map $h:P\to X_{(\theta,\omega)}$ is also \emph{surjective} submersion. In general, the classifying map $h: P \to X$ of a $G$-realization $H:TP\to A$ is a submersion onto an \emph{open} $G$-saturated subset of a leaf of $A$ -- see \cite[Lemma 3.29]{FS19}

\subsection{Cartan's realization problem}

Many classification problems in geometry can be formulated as follows: 
\begin{problem}[{\bf Cartan's Realization Problem}]
\label{prob:Cartan}
One is given \textbf{Cartan Data}:
\begin{itemize}
\item a closed Lie subgroup $G \subset \GL(n,\Rr)$;
\item a $G$-manifold $X$ with infinitesimal action $\psi:X\times\gg\to TX$;
\item equivariant maps $T: X \to \hom(\wedge^2\Rr^n, \Rr^n)$, $R: X \to \hom(\wedge^2\Rr^n, \gg)$ and $F: X \times \Rr^n \to TX$;
\end{itemize}
and asks for the existence of \textbf{solutions}: 
\begin{itemize}
\item an n-dimensional effective orbifold $M$;
\item a $G$-structure $\B_G(M)\to M$ with tautological form $\theta\in\Omega^1(\B_G(M), \Rr^n)$ and connection 1-form $\omega \in \Omega^1(\B_G(M), \gg)$;
\item an equivariant map $h: \B_G(M) \to X$;
\end{itemize} 
satisfying the structure equations:
\begin{equation*}
\left\{
\begin{array}{l}
\d \theta=T(h)( \theta \wedge \theta) - \omega \wedge \theta\\
\d \omega= R (h)( \theta \wedge \theta) 
- \omega\wedge\omega
\\
\d h = F(h,\theta) + \psi(h, \omega)
\end{array}
\right.
\end{equation*}
\end{problem}

One can show that such a problem has a solution for each $x\in X$ if and only if the Cartan data defines a $G$-structure algebroid with connection, in canonical form, on the trivial bundle $X\times(\Rr^n\oplus\gg)\to X$, as in Section \ref{sec:canonical:form:2}. Moreover, a solution is just a $G$-realization of this algebroid. 

\begin{theorem}[\cite{FS19}]
Given Cartan data $(G,X,T,R,F)$ there is a solution to Cartan's Realization Problem for each $x\in X$ if and only if the data defines a $G$-structure algebroid with connection on $A=X\times(\Rr^n\oplus\gg)$. Moreover, there is a 1:1 correspondence:
\begin{equation*}
\left\{\txt{solutions\\ \\ $(\B_G(M),\theta,\omega,h)$\ \\ \,}\right\}
\tilde{\longleftrightarrow}
\left\{\txt{\quad $G$-realizations \qquad \\ \\ \hskip -30 pt $$\xymatrix{TP\ar[r]^---{(\theta,\omega)} & A}$$\qquad \\ \,} \right\}.
\end{equation*}
\end{theorem}

In this paper, we are only interested in the realization problem for the classifying algebroid of a fully regular $G$-structure with connection. We refer to \cite{FS19} for a much wider discussion of Cartan's Realization Problem and its solutions using the theory $G$-principal groupoids.

\subsection{$G$-realizations of a classifying Lie algebroid}

Let $(P,\theta,\omega)$ be a fully regular $G$-structure with connection. We can viewed it as a $G$-realization of its classifying Lie algebroid:
\[
\xymatrix@C=15pt{TP\ar[d] \ar[rr]^---{H=(\theta,\omega)} & & A_{(\theta,\omega)} \ar[d]\\
P \ar[rr]_h && X_{(\theta,\omega)}}
\]
Any other $G$-realization is locally equivalent to this one:

\begin{prop}
Let $(P,\theta,\omega)$ be a fully regular $G$-structure with connection. Given any $G$-realization 
\[
\xymatrix@C=15pt{TP'\ar[d] \ar[rr]^---{H'} & & A_{(\theta,\omega)} \ar[d]\\
P' \ar[rr]_{h'} && X_{(\theta,\omega)}}
\]
then for any $p'_0\in P'$ and $p_0\in P$ with $h'(p_0')=h(p_0)$ there is a local equivalence $\Phi:P'\to P$ such that $\Phi(p_0')=p_0$.
\end{prop}

\begin{proof}
The result follows from \cite[Thm 7.2]{FS19} applied to the classifying algebroid $A_{(\theta,\omega)}$.
\end{proof}

According to Example \ref{ex:G:realization:groupoid}, any $G$-integration of $A_{(\theta,\omega)}$ gives rise to $G$-realizations by considering the source fibers of the integration. The previous result shows that these $G$-realizations are locally equivalent to $(P,\theta,\omega)$ and we will see next paragraph that they exhibit some remarkable properties that distinguish them from arbitrary fully regular $G$-structures.  

Starting with Cartan data $(G,X,T,R,F)$ defining a $G$-structure algebroid with connection $A\to X$, as explained above, the solutions of the corresponding Cartan's realization problem need not be fully regular $G$-structures with connection. For example, any $G$-structure with connection $(P,\theta,\omega)$ can be viewed as a realization of $A=TP\to P$, which is a  $G$-structure algebroid with connection (cf.~Example \ref{ex:principal:bundle:algebroid}). However, when a solution is fully regular we have the following relationship between the Lie algebroid $A$ and the classifying Lie algebroids of the solution:

\begin{prop}
Let $(G,X,T,R,F)$ be Cartan data defining a $G$-structure algebroid with connection $A\to X$. If a solution $(\B_G(M),\theta,\omega,h)$ to the realization problem is fully regular then there is a morphism $(\Phi,\phi)$ of $G$-structure algebroids with connection:
\[
 \xymatrix{
 & T\B_G(M) \ar[dl] \ar[dr]\ar@{-}[d] & \\
A|_U\ar[dd] \ar[rr]^{\Phi\qquad} &\ar[d] & A_{(\theta,\omega)}\ar[dd] \\
 & \B_G(M)  \ar[dl] \ar[dr] & \\
U\ar[rr]^{\phi\qquad} & & X_{(\theta,\omega)}}
\]
where $U=\Im h$ is an open in a leaf of $A$, $\phi:U\to X_{(\theta,\omega)}$ is a surjective submersion, and $\Phi$ is a fiberwise isomorphism.
\end{prop}

\begin{proof}
It follows from \cite[Cor 5.11]{FernandesStruchiner1} that there is a morphism $(\Phi,\phi)$ of Lie algebroids as in the diagram, making it commute. This is a morphism of $G$-structure algebroids with connection because both maps $T\B_G(M)\to A$ and $T\B_G(M)\to A_{(\theta,\omega)}$ are fiberwise isomorphisms of $G$-structure algebroids with connection and $\Phi$ is a fiberwise isomorphism.
\end{proof}

\subsection{Complete $G$-realizations and integrability}
In general, a $G$-structure algebroid with connection fails to be $G$-integrable. 

A $G$-realization $H:TP\to A$ of a $G$-structure algebroid with connection is called a \textbf{full realization} of $A$ if $h:P\to X$ is surjective onto a leaf of $A$. One could be lead to believe that the existence of a full $G$-realization covering a leaf $L\subset X$ implies $G$-integrability of the restricted Lie algebroid $A|_L$. The next example shows that this is not the case.

\begin{example}
Let $X$ be a manifold with a closed 2-form $\omega\in\Omega^2(X)$. The extension Lie algebroid of $TX$ associated with $\omega$ has underlying vector bundle
\[ A:=TX\oplus \Rr, \]
anchor the projection $\rho:TM\oplus\Rr\to TM$, and Lie bracket defined by:
\[ [(V,f),(W,g)]_A:=([V,W],\Lie_Vg -\Lie_W f+\omega(V,W)), \]
where $V,W\in\X(X)$ and $f,g\in C^\infty(X)$.
It is well-known that this Lie algebroid is integrable if and only if the group of spherical periods of $\omega$:
\[ \Per(\omega):=\Big\{\int_\gamma \omega:\gamma\in\pi_2(X)\Big\}\subset \Rr, \]
is a discrete subgroup of $\Rr$ -- see Example 3.28 in \cite{CrainicFernandes:lectures}. 

If we assume that $X$ is parallelizable then $A$ is a trivial vector bundle, hence it is a $\{e\}$-structure algebroid. Assume further that one has a manifold $Y$, where $\pi_2(Y)=0$, together with a surjective local diffeomorphism:
\[ \phi:Y\to X, \]
Note that $Y$ is also parallelizable. Then the pullback bundle $\phi^*A=TY\oplus\Rr$ is also an extension Lie algebroid for the pullback form $\phi^*\omega$, which is now an integrable $\{e\}$-structure algebroid. Note that this Lie algebroid structure makes the pullback square
\[
\xymatrix{
\phi^*A\ar[d]\ar[r]^{\Phi} & A\ar[d]\\
Y\ar[r]_{\phi} & X}
\]
a Lie algebroid morphism. Then any integration $\G\tto Y$ of $\phi^*A$ (which is obviously an $\{e\}$-integration) gives a realization $(\s^{-1}(y),\wmc|_{\s^{-1}(y)})$ of $\phi^*A$. Composing this realization with the morphism $(\phi,\Phi)$ yields a full realization of $A$. 

For a concrete example, one can take $X=(\Rr^3-\{0\})\times(\Rr^3-\{0\})$ with closed 2-form:
\[ \omega=\pr_1^*\eta+\sqrt{2}\, \pr_2^*\eta, \]
where $\eta$ denotes a closed 2-form on $\Rr^3-\{0\}$ whose integral over $\Ss^2$ is non-zero. We leave as an exercise to check that there exists $Y$, with $\pi_2(Y)=0$, and a surjective, local diffeomorphism $\phi:Y\to X$.
\end{example}

In order to obtain $G$-integrability, one needs full $G$-realizations which are also \emph{complete}, as we now explain. 

Given a $G$-realization $H:TP\to A$ of a $G$-structure algebroid with connection one obtains a Lie algebroid action along the base map $h:P\to X$ by setting:
\[ \sigma:h^*A\to TP, \quad \sigma(p,a)):=H_p^{-1}(a). \]
Indeed, using the fact that $H$ is a Lie algebroid map, one checks that the two defining properties of a Lie algebroid action hold:
\begin{enumerate}[(i)]
\item $\d h\circ \sigma=\rho_A$;
\item $\sigma([s_1,s_2]_A)=[\sigma(s_1),\sigma(s_2)]$, for all sections $s_1,s_2\in\Gamma(A_{(\theta,\omega)})$.
\end{enumerate}

We recall that an action $\sigma$ of a Lie algebroid $A\to X$ on a map $h:P\to X$ is called a \textbf{complete action} if for any compactly supported section $s\in\Gamma(A)$ the vector field $\sigma(s)\in\X(P)$ is complete. This leads to the following:

\begin{definition}
A $G$-realization $H:TP\to A$ of a $G$-structure algebroid with connection is called a \textbf{complete realization} of $A$ if it is a full realization whose associated  Lie algebroid action $\sigma:h^*A\to TP$ is complete.
\end{definition}

In particular, a fully regular $G$-structure with connection $(P,\theta,\omega)$ is automatically a full realization of its classifying algebroid $A_{(\theta,\omega)}$ and we will call it \textbf{complete} if is also a complete realization of $A_{(\theta,\omega)}$.

\begin{example}
Let $A_{(\theta,\omega)}$ be the classifying Lie algebroid of a fully regular $O(n)$-structure $(P,\theta,\omega)$, where $\omega$ is the Levi-Civita connection. Then this is a complete $G$-realization if and only if the corresponding metric on $M=P/O(n)$ is a complete metric -- see \cite[Thm 8.5]{FS19}. 
\end{example}

\begin{example}
\label{ex:G:realization:s:fiber}
Let $\G\tto X$ be a $G$-structure groupoid with connection integrating $A_{(\theta,\omega)}$. Then the $G$-realizations $(\s^{-1}(x),\wmc|_{\s^{-1}(x)})$ furnish examples of complete realizations. The reason is that the vector field $\sigma(s)$ coincides with the restriction of the right-invariant vector field generated by the section $s$, and these are complete whenever $\rho(s)$ is complete. 
\end{example}

We have the following converse to the previous example -- see \cite[Thm 6.9]{FS19}:

\begin{theorem}
Let $(P,\theta,\omega)$ be a fully regular $G$-structure with connection. The classifying $G$-structure algebroid $A_{(\theta,\omega)}$ is $G$-integrable if and only if it admits a complete $G$-realization.
\end{theorem}

Notice that for $A_{(\theta,\omega)}$ to be $G$-integrable,  the fully regular $G$-structure with connection $(P,\theta,\omega)$ does not need to be complete. For example, a non-complete metric of constant scalar curvature gives rise to a non-complete fully regular $\O(n)$-structure with connection whose classifying Lie algebroid is $G$-integrable.

The $G$-realizations arising as $\s$-fibers from $G$-integrations, as in Example \ref{ex:G:realization:s:fiber}, enjoy even more special properties. For example, every germ of infinitesimal symmetry $\xi\in\X(\s^{-1}(x),\wmc)_p$ is the restriction of an infinitesimal symmetry in $\X(\s^{-1}(x),\wmc)$, and every infinitesimal symmetry is a complete vector field. This motivates the following definition:

\begin{definition}\label{def:SC}
A $G$-realization $(P, (\theta, \omega), h)$ of a $G$-structure algebroid $A$ with connection is called \textbf{strongly complete} if $h:P\to X$ is surjective onto a leaf of $A$ and any local symmetry of $(P, (\theta, \omega), h)$ extends to a global symmetry.
\end{definition}

The property of being strongly complete characterizes the $G$-realizations arising from source fibers of $G$-integrations. In particular, this shows that they are complete realizations. Here we are interested in the special case of the classifying Lie algebroid of a fully regular $G$-structure with connection:

\begin{theorem} 
\label{thm:sym:coframe}
Let $(P,\theta,\omega)$ be a fully regular $G$-structure with connection. If $(P,\theta,\omega)$ is strongly complete then the group of symmetries $\Diff(P,\theta,\omega)$ is a Lie group with Lie algebra $\X(P,\theta,\omega)$ isomorphic to the isotropy Lie algebras of $A_{(\theta,\omega)}$. Moreover, the natural action of $\Diff(P,\theta,\omega)$ on $P$ commutes with the $G$-action, and is a proper and free action whose orbits coincide with the fibers of the classifying map $h$. In particular,  $h:P\to X_{(\theta,\omega)}$ is a principal $\Diff(P,\theta,\omega)$-bundle.
\end{theorem}

This result is a consequence of \cite[Prop 6.12]{FS19}. The fact that a strongly complete fully regular $G$-structure with connection $(P,\theta,\omega)$ is isomorphic to the source fiber of a $G$-integration can be seen as follows. By the previous theorem, $h:P\to X_{(\theta,\omega)}$ is a principal $H:=\Diff(P,\theta,\omega)$-bundle. Hence, we have the Atiyah algebroid:
\[ TP/H\to  X_{(\theta,\omega)}. \]
The coframe $(\theta,\omega)$ induces a Lie algebroid isomorphism of this Atiyah algebroid with $A_{(\theta,\omega)}$ and so its gauge groupoid
\[ \G:=(P\times P)/H\tto  X_{(\theta,\omega)} \]
is a $G$-structure groupoid with connection integrating $A_{(\theta,\omega)}$. Its source fibers are isomorphic to $(P,\theta,\omega)$. 

If $(P,\theta,\omega)$ is a complete fully regular $G$-structure with connection one can construct a ``larger'' fully regular $G$-structure with connection, locally isomorphic to $(P,\theta,\omega)$, and strongly complete. For that, one considers the universal covering space $\widetilde{M}$ of the effective orbifold $M=P/G$. Then one has a pullback $G$-structure:
\[
\xymatrix{
\widetilde{P}\ar[d]\ar[r]^q& P\ar[d]\\
\widetilde{M}\ar[r] & M
}
\]
This yields a complete $G$-realization $(\tilde{P},(q^*\theta,q^*\omega),q^*h)$ of $A_{(\theta,\omega)}$. Using that $\tilde{M}=\tilde{P}/G$ is now an orbifold with trivial orbifold fundamental group, one can show that this realization must be strongly complete, and is actually isomorphic to an $\s$-fiber of the canonical $G$-integration -- see \cite[Prop 6.10]{FS19}. In conclusion, we have the following corollary of Theorem \ref{thm:sym:coframe}:

\begin{corol}
Every complete fully regular $G$-structure with connection $(P,\theta,\omega)$ is covered by a $\s$-fiber of the canonical $G$-integration of its classifying Lie algebroid.
\end{corol}

\begin{example}
It is now easy to explain why we did not find the flat metric on the torus $\Tt^n$ in the discussion in Example \ref{ex:curvature:2}. Although this is a complete metric and the associated $\SO_n(\Rr)$-structure $(P,\theta,\omega)$ is fully regular and complete, it is not strongly complete: their are local isometries of the torus which do not extend to global isometries. On the other hand, the $\SO_n(\Rr)$-structure associated with Euclidean flat space provides a cover of $(P,\theta,\omega)$ which is strongly complete and indeed the $\s$-fiber of the canonical $\SO_n(\Rr)$-integration of classifying Lie algebroid of the flat torus.
\end{example}

More generally, one can even classify all fully regular $G$-structures with connection whose classifying  Lie  algebroid is $G$-integrable. 

\begin{theorem}
Let $(P,\theta,\omega)$ be a fully regular $G$-structure with connection and  suppose that  the  classifying  Lie  algebroid $A_{(\theta,\omega)}$ is $G$-integrable.  Then $(P,\theta,\omega)$ is covered by a full $G$-realization of $A_{(\theta,\omega)}$ which covers an open $G$-invariant subset of an $\s$-fiber of the canonical $G$-integration.
\end{theorem}

\begin{remark}
In the previous theorem, by a ``cover'' of $P$ we mean a manifold $Q$ together with a surjective local diffeomorphism $\phi:Q\to P$ which is not necessarily an even cover.
\end{remark}

\begin{proof}
Let $\G:=\Sigma_G(A_{(\theta,\omega)})\tto X_{(\theta,\omega)}$ be the canonical $G$-integration. As in the proof of Theorem 6.8 of \cite{FernandesStruchiner1} one considers the distribution $\cD$ in $P\tensor[_h]{\times}{_\t} \G$ given by:
\[ \cD:=\ker\Big(\pr^*_P(\theta,\omega)-\pr^*_\G\wmc\Big). \]
This is an involutive distribuition and one can take any maximal integral submanifold $Q$ of $\cD$. The proof in \cite{FernandesStruchiner1} shows that:
\begin{enumerate}[(a)]
\item The projection $\pr_P:Q\to P$ is a surjective local diffeomorphism;
\item The projection $\pr_\G:Q\to \G$ is a local diffeomorphism onto an open subset $U$ of an $\s$-fiber;
\item The restriction $(\widetilde\theta,\widetilde{\omega}):=\pr^*_P(\theta,\omega)|_Q=\pr^*_\G\wmc|_Q$ gives a coframe in $Q$ making both projections equivalences.
\end{enumerate}
We observe that $Q$ is invariant under the (diagonal) action of the connected component of the identity $G^0$, since the infinitesimal action is tangent to $\cD$. The components $\widetilde\theta$ and $\widetilde{\omega}$ of the coframe then give an $\Rr^n$-valued and a $\gg$-valued form that satisfy the properties of a tautological and a connection form with structure group $G^0$.

If $G$ is connected, the result follows. If $G$ is not connected, then one needs to take the $G$-saturations of $Q$ and $U$, which then satisfy the conditions in the statement of the theorem. Note that in the latter case both the resulting cover and the open subset of the $\s$-fiber may be disconnected, even if $P$ is connected (our standing assumption).
\end{proof}

The previous result motivates introducing the following notion: Two $G$-structures with connection $(P_1,\theta_1,\omega_1)$ and $(P_2,\theta_2,\omega_2)$ are called \textbf{globally equivalent up to cover} if there exists a $G$-structure with connection $(Q,\widetilde{\theta},\widetilde{\omega})$ and surjective local diffeomorphisms:
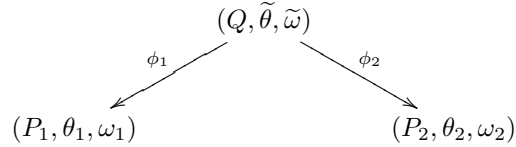
\[
\xymatrix{
 & (Q,\widetilde{\theta},\widetilde{\omega})\ar[dl]_{\phi_1}\ar[dr]^{\phi_2} & \\
(P_1,\theta_1,\omega_1) & & (P_2,\theta_2,\omega_2)}
\]
such that $\phi^*_i\theta_i=\widetilde{\theta}$ and $\phi^*_i\omega_i=\widetilde{\omega}$. Then, when the rank of the fully regular coframe attains its smallest value, one obtains:

\begin{corol}
Let $(P,\theta,\omega)$ be a fully regular $G$-structure with connection of rank 0. Then $P$ is locally equivalent up to cover to an open subset of the ``space form" $G$-structure of Example \ref{ex:space:form}.
\end{corol}

On the other hand, when the rank of the fully regular coframe attains its largest value, one finds that:

\begin{corol}
Let $(P,\theta,\omega)$ be a fully regular $G$-structure with connection of rank $n+\dim\gg$. Then $X_{(\theta,\omega)}$ is itself a $G$-structure with connection covered by $(P,\theta,\omega)$.
\end{corol}


\bibliographystyle{abbrv}
\bibliography{bibliography2}

\end{document}